\documentclass[reqno]{amsart}

\usepackage[all]{xy}
\usepackage{graphicx}

\usepackage{amssymb}
\usepackage{amsmath}
\usepackage{mathrsfs}
\usepackage{epsfig}
\usepackage{amscd}
\usepackage{graphicx, color}

\def\E{\ifmmode{\mathbb E}\else{$\mathbb E$}\fi} 
\def\N{\ifmmode{\mathbb N}\else{$\mathbb N$}\fi} 
\def\R{\ifmmode{\mathbb R}\else{$\mathbb R$}\fi} 
\def\Q{\ifmmode{\mathbb Q}\else{$\mathbb Q$}\fi} 
\def\C{\ifmmode{\mathbb C}\else{$\mathbb C$}\fi} 
\def\H{\ifmmode{\mathbb H}\else{$\mathbb H$}\fi} 
\def\Z{\ifmmode{\mathbb Z}\else{$\mathbb Z$}\fi} 
\def\P{\ifmmode{\mathbb P}\else{$\mathbb P$}\fi} 
\def\T{\ifmmode{\mathbb T}\else{$\mathbb T$}\fi} 
\def\SS{\ifmmode{\mathbb S}\else{$\mathbb S$}\fi} 
\def\DD{\ifmmode{\mathbb D}\else{$\mathbb D$}\fi} 

\renewcommand{\d}{\delta}

\newcommand{\e}{\varepsilon}

\newcommand{\del}{\partial}

\newcommand{\ben}{\begin{enumerate}}
\newcommand{\een}{\end{enumerate}}
\newcommand{\be}{\begin{equation}}
\newcommand{\ee}{\end{equation}}
\newcommand{\bea}{\begin{eqnarray}}
\newcommand{\eea}{\end{eqnarray}}
\newcommand{\beastar}{\begin{eqnarray*}}
\newcommand{\eeastar}{\end{eqnarray*}}
\newcommand{\bc}{\begin{center}}
\newcommand{\ec}{\end{center}}

\newcommand{\rf}[1]{\ensuremath{{#1}_{\alpha}}}

\newcommand{\coker}{\textrm{coker}}

\newcommand{\leng}{\operatorname{leng}}
\newcommand{\dist}{\operatorname{dist}}

\theoremstyle{theorem}
\newtheorem{thm}{Theorem}[section]
\newtheorem{cor}[thm]{Corollary}
\newtheorem{lem}[thm]{Lemma}
\newtheorem{prop}[thm]{Proposition}

\theoremstyle{definition}
\newtheorem{defn}[thm]{Definition}
\newtheorem{rem}[thm]{Remark}

\newtheorem*{thm*}{Theorem}

\numberwithin{equation}{section}

\def\R{{\mathbb R}}

\def\Crit{{\hbox{Crit}}}

\def\E{{\mathbb E}}
\def\Z{{\mathbb Z}}
\def\C{{\mathbb C}}
\def\R{{\mathbb R}}
\def\P{{\mathbb P}}

\def\N{{\mathbb N}}

\def\11{{\mathbb I}}

\def\delbar{{\bar \partial}}
\def\dudtau{{\frac{\del u}{\del \tau}}}
\def\dudt{{\frac{\del u}{\del t}}}

\def\area{{\operatorname{area}}}

\def\C{\mathbb{C}}
\def\Z{\mathbb{Z}}

\def\T{\mathbb{T}}

\def\Q{\mathbb{Q}}

\def\E{\ifmmode{\mathbb E}\else{$\mathbb E$}\fi} 
\def\N{\ifmmode{\mathbb N}\else{$\mathbb N$}\fi} 
\def\R{\ifmmode{\mathbb R}\else{$\mathbb R$}\fi} 
\def\Q{\ifmmode{\mathbb Q}\else{$\mathbb Q$}\fi} 
\def\C{\ifmmode{\mathbb C}\else{$\mathbb C$}\fi} 
\def\H{\ifmmode{\mathbb H}\else{$\mathbb H$}\fi} 
\def\Z{\ifmmode{\mathbb Z}\else{$\mathbb Z$}\fi} 
\def\P{\ifmmode{\mathbb P}\else{$\mathbb P$}\fi} 
\def\SS{\ifmmode{\mathbb S}\else{$\mathbb S$}\fi} 
\def\DD{\ifmmode{\mathbb D}\else{$\mathbb D$}\fi} 

\def\R{{\mathbb R}}

\def\Crit{{\hbox{Crit}}}
\def\E{{\mathbb E}}
\def\Z{{\mathbb Z}}
\def\C{{\mathbb C}}
\def\R{{\mathbb R}}

\def\N{{\mathbb N}}
\def\MM{{\mathcal M}}

\def\delbar{{\overline \partial}}

\def\d{\delta}  
\def\e{\varepsilon}

\def\U{\Upsilon}
\def\CA{{\mathcal A}}
\def\CB{{\mathcal B}}
\def\CC{{\mathcal C}}
\def\CD{{\mathcal D}}
\def\CE{{\mathcal E}}
\def\CF{{\mathcal F}}
\def\CG{{\mathcal G}}
\def\CH{{\mathcal H}}
\def\CI{{\mathcal I}}
\def\CJ{{\mathcal J}}
\def\CK{{\mathcal K}}
\def\CL{{\mathcal L}}
\def\CM{{\mathcal M}}
\def\CN{{\mathcal N}}
\def\CO{{\mathcal O}}
\def\CP{{\mathcal P}}
\def\CQ{{\mathcal Q}}
\def\CR{{\mathcal R}}
\def\CP{{\mathcal P}}
\def\CS{{\mathcal S}}
\def\CT{{\mathcal T}}
\def\CU{{\mathcal U}}
\def\CV{{\mathcal V}}
\def\CW{{\mathcal W}}
\def\CX{{\mathcal X}}
\def\CY{{\mathcal Y}}
\def\CZ{{\mathcal Z}}
\def\EJ{\mathfrak{J}}
\def\EM{\mathfrak{M}}
\def\ES{\mathfrak{S}}
\def\BB{\mathib{B}}
\def\BC{\mathib{C}}
\def\BH{\mathib{H}}

\def\rd{\partial}
\def\grad#1{\,\nabla\!_{{#1}}\,}
\def\gradd#1#2{\,\nabla\!_{{#1}}\nabla\!_{{#2}}\,}
\def\om#1#2{\omega^{#1}{}_{#2}}
\def\vev#1{\langle #1 \rangle}
\def\darr#1{\raise1.5ex\hbox{$\leftrightarrow$}
\mkern-16.5mu #1}
\def\Ha{{1\over2}}
\def\ha{{\textstyle{1\over2}}}
\def\fr#1#2{{\textstyle{#1\over#2}}}
\def\Fr#1#2{{#1\over#2}}
\def\rf#1{\fr{\rd}{\rd #1}}
\def\rF#1{\Fr{\rd}{\rd #1}}
\def\df#1{\fr{\d}{\d #1}}
\def\dF#1{\Fr{\d}{\d #1}}
\def\DDF#1#2#3{\Fr{\d^2 #1}{\d #2\d #3}}
\def\DDDF#1#2#3#4{\Fr{\d^3 #1}{\d #2\d #3\d #4}}
\def\ddF#1#2#3{\Fr{\d^n#1}{\d#2\cdots\d#3}}
\def\fs#1{#1\!\!\!/\,}   
\def\Fs#1{#1\!\!\!\!/\,} 
\def\roughly#1{\raise.3ex\hbox{$#1$\kern-.75em
\lower1ex\hbox{$\sim$}}}
\def\ato#1{{\buildrel #1\over\longrightarrow}}
\def\up#1#2{{\buildrel #1\over #2}}
\def\opname#1{\mathop{\kern0pt{\rm #1}}\nolimits}
\def\Re{\opname{Re}}
\def\Im{\opname{Im}}
\def\End{\opname{End}}
\def\dim{\opname{dim}}

\def\Area{\opname{Area}}

\def\group#1{\opname{#1}}
\def\SU{\group{SU}}
\def\U{\group{U}}
\def\SO{\group{SO}}
\def\pr{\prime}
\def\ppr{{\prime\prime}}
\def\bs{\mathib{s}}
\def\supp{\operatorname{supp}}

\def\Dev{\operatorname{Dev}}

\def\leng{\operatorname{leng}}
\def\Per{\operatorname{Per}}
\def\End{\operatorname{End}}

\def\coker{\operatorname{Coker}}

\begin{document}
\quad \vskip1.375truein

\def\mq{\mathfrak{q}}
\def\mp{\mathfrak{p}}
\def\mH{\mathfrak{H}}
\def\mh{\mathfrak{h}}
\def\ma{\mathfrak{a}}
\def\ms{\mathfrak{s}}
\def\mm{\mathfrak{m}}
\def\mn{\mathfrak{n}}
\def\mz{\mathfrak{z}}
\def\mw{\mathfrak{w}}
\def\Hoch{{\tt Hoch}}
\def\mt{\mathfrak{t}}
\def\ml{\mathfrak{l}}
\def\mT{\mathfrak{T}}
\def\mL{\mathfrak{L}}
\def\mg{\mathfrak{g}}
\def\md{\mathfrak{d}}
\def\mr{\mathfrak{r}}

\def\CA{{\mathcal A}}
\def\CB{{\mathcal B}}
\def\CC{{\mathcal C}}
\def\CD{{\mathcal D}}
\def\CE{{\mathcal E}}
\def\CF{{\mathcal F}}
\def\CG{{\mathcal G}}
\def\CH{{\mathcal H}}
\def\CI{{\mathcal I}}
\def\CJ{{\mathcal J}}
\def\CK{{\mathcal K}}
\def\CL{{\mathcal L}}
\def\CM{{\mathcal M}}
\def\CN{{\mathcal N}}
\def\CO{{\mathcal O}}
\def\CP{{\mathcal P}}
\def\CQ{{\mathcal Q}}
\def\CR{{\mathcal R}}
\def\CP{{\mathcal P}}
\def\CS{{\mathcal S}}
\def\CT{{\mathcal T}}
\def\CU{{\mathcal U}}
\def\CV{{\mathcal V}}
\def\CW{{\mathcal W}}
\def\CX{{\mathcal X}}
\def\CY{{\mathcal Y}}
\def\CZ{{\mathcal Z}}
\def\EJ{\mathfrak{J}}
\def\EM{\mathfrak{M}}
\def\ES{\mathfrak{S}}
\def\BB{\mathib{B}}
\def\BC{\mathib{C}}
\def\BH{\mathib{H}}

\def\rd{\partial}
\def\grad#1{\,\nabla\!_{{#1}}\,}
\def\gradd#1#2{\,\nabla\!_{{#1}}\nabla\!_{{#2}}\,}
\def\om#1#2{\omega^{#1}{}_{#2}}
\def\vev#1{\langle #1 \rangle}
\def\darr#1{\raise1.5ex\hbox{$\leftrightarrow$}
\mkern-16.5mu #1}
\def\Ha{{1\over2}}
\def\ha{{\textstyle{1\over2}}}
\def\fr#1#2{{\textstyle{#1\over#2}}}
\def\Fr#1#2{{#1\over#2}}
\def\rf#1{\fr{\rd}{\rd #1}}
\def\rF#1{\Fr{\rd}{\rd #1}}
\def\df#1{\fr{\d}{\d #1}}
\def\dF#1{\Fr{\d}{\d #1}}
\def\DDF#1#2#3{\Fr{\d^2 #1}{\d #2\d #3}}
\def\DDDF#1#2#3#4{\Fr{\d^3 #1}{\d #2\d #3\d #4}}
\def\ddF#1#2#3{\Fr{\d^n#1}{\d#2\cdots\d#3}}
\def\fs#1{#1\!\!\!/\,}   
\def\Fs#1{#1\!\!\!\!/\,} 
\def\roughly#1{\raise.3ex\hbox{$#1$\kern-.75em
\lower1ex\hbox{$\sim$}}}
\def\ato#1{{\buildrel #1\over\longrightarrow}}
\def\up#1#2{{\buildrel #1\over #2}}
\def\opname#1{\mathop{\kern0pt{\rm #1}}\nolimits}
\def\Re{\opname{Re}}
\def\Im{\opname{Im}}
\def\End{\opname{End}}
\def\dim{\opname{dim}}
\def\vol{\opname{vol}}
\def\group#1{\opname{#1}}
\def\SU{\group{SU}}
\def\U{\group{U}}
\def\SO{\group{SO}}
\def\pr{\prime}
\def\ppr{{\prime\prime}}
\def\bs{\mathib{s}}
\def\dudtau{\frac{\partial u}{\partial \tau}}
\def\dudt{\frac{\partial u}{\partial t}}
\def\dvdtau{\frac{\partial v}{\partial \tau}}
\def\dvdt{\frac{\partial v}{\partial t}}
\def\supp{\operatorname{supp}}
\def\Cal{\operatorname{Cal}}
\def\coker{\operatorname{coker}}
\def\diam{\operatorname{diam}}
\def\area{\operatorname{area}}
\def\Graph{\operatorname{Graph}}
\def\Dev{\operatorname{Dev}}
\def\Fix{\operatorname{Fix}}
\def\curv{\operatorname{curv}}
\def\Spec{\operatorname{Spec}}
\def\dist{\operatorname{dist}}
\def\Crit{\operatorname{Crit}}
\def\codim{\operatorname{codim}}
\def\Per{\operatorname{Per}}
\def\leng{\operatorname{leng}}
\def\Area{\operatorname{Area}}
\def\Image{\operatorname{Image}}


\def\mq{\mathfrak{q}}
\def\mH{\mathfrak{H}}
\def\mh{\mathfrak{h}}
\def\ma{\mathfrak{a}}
\def\ms{\mathfrak{s}}
\def\mm{\mathfrak{m}}
\def\mn{\mathfrak{n}}

\def\Hoch{{\tt Hoch}}
\def\mt{\mathfrak{t}}
\def\ml{\mathfrak{l}}
\def\mT{\mathfrak{T}}
\def\mL{\mathfrak{L}}
\def\mg{\mathfrak{g}}
\def\md{\mathfrak{d}}

\title[Localization of Floer complex]{
Localization of Floer homology of engulfable topological
Hamiltonian loop}

\author{Yong-Geun Oh}
\thanks{This work is supported by the Institute for Basic Sciences and partially
supported by US NSF grant \# DMS 0904197}

\address{ IBS Center for Geometry and Physics, Institute for Basic Sciences,
Pohang, Korea \& Department of Mathematics, POSTECH, Pohang, KOREA,
\& Department of Mathematics, University of Wisconsin, Madison, WI
53706}
\email{oh@math.wisc.edu}

\begin{abstract}
Localization of Floer homology is first introduced by Floer \cite{floer:fixed} in the
context of Hamiltonian Floer homology. The author
employed the notion in the Lagrangian context for the
pair $(\phi_H^1(L),L)$ of compact Lagrangian submanifolds in
tame symplectic manifolds $(M,\omega)$ in \cite{oh:newton,oh:imrn}
for a compact Lagrangian submanifold $L$ and $C^2$-small Hamiltonian $H$.
In this article, motivated by the study of topological Hamiltonian dynamics,
we extend the localization process
for any engulfable Hamiltonian path $\phi_H$ whose time-one map $\phi_H^1$
is sufficiently $C^0$-close to the identity (and also to the case of triangle
product), and prove that the value of local Lagrangian spectral invariant
is the same as that of global one. Such a Hamiltonian path naturally
occurs as an approximating sequence of engulfable topological Hamiltonian loop.
We also apply this localization to the graphs $\Graph \phi_H^t$ in
$(M\times M, \omega\oplus -\omega)$ and localize the Hamiltonian Floer
complex of such a Hamiltonian $H$. We expect that this study will play an important role in the study of
homotopy invariance of the spectral invariants of topological Hamiltonian.
\end{abstract}

\keywords{Local Floer homology, engulfable topological Hamiltonian loop,
$J_0$-convex domain, maximum principle, thick-thin decomposition, handle sliding lemma}

\date{November 14, 2011; revised on May 28, 2013}

\maketitle

\hskip0.3in MSC2010: 53D05, 53D35, 53D40; 28D10.
\medskip

\tableofcontents

\section{Introduction and the main results}
\label{sec:intro}

The construction of the local version of the Floer homology was introduced by
Floer \cite{floer:fixed}. The present author applied this construction to the
Lagrangian context and defined the local Floer homology, denoted by $HF(H,L;U)$,
which singles out the contribution from the Floer trajectories whose images
are contained in a given Darboux neighborhood $U$ of $L$ in $M$.
Such an isolation of the contribution is proven to be possible
and the resulting Floer homology is isomorphic to the singular homology $H_*(L)$ (with $Z_2$-coefficients)
in \cite{oh:imrn}, \emph{ provided $H$ is $C^2$-small}. This $C^2$-smallness
is used, conspicuously in \cite{oh:imrn}, so that first
\be\label{eq:isolated}
\phi_{H^t}(L) \subset V \subset \overline V \subset U
\ee
holds for all $t \in [0,1]$, and then the `thick-thin'
decomposition of the Floer trajctories exists. The necessity of such a decomposition
is highlighted for the Floer moduli spaces for the boundary map, but
its necessity is less conspicuous for that of the chain map in \cite{oh:imrn}:

But this latter was further scrutinized and exploited by Chekanov in his
study of displacement energy in \cite{chekanov:newton,chekanov:area}.
It follows from his argument in \cite{chekanov:newton} that the quasi-isomorphism
property of thin part of Floer chain maps between the local Floer complex
$H$ and the Morse complex of $f$ holds for a sufficiently small $\e> 0$
as long as $\|H\| < \frac{1}{2} A(M,L,J_0)$ \emph{as long as the thick-thin decomposition
exists for the chain map}. Here  $A(M,L,J_0)$ is  the smallest area of the non-constant
$J_0$-holomorphic spheres or discs attached to $L$.
(Chekanov denotes $\sigma(M,L,J_0)$ instead of $A(M,L,J_0)$.)
The required thick-thin decomposition was established via the thick-thin
decomposition of associated Floer moduli spaces into those with big areas
and those with very small areas. It was proved in \cite{oh:imrn} that this
dichotomy exists when $H$ is $C^2$-small by proving that all the
thick trajectories have symplectic area greater than, say $\frac{1}{2} A(M,L,J_0)$,
by a variation of Gromov-Floer compactness as $\phi_H^1(L) \to L$ in $C^1$-topology
(or $H \to 0$ in $C^2$-topology). (We would like to emphasize
that this convergence argument is not the standard Gromov-Floer type
compactness argument since the limiting configuration is degenerate.
The precise study of this convergence belongs to the realm of the so called adiabatic
limit in the sense of \cite{foh:ajm,oh:newton,oh:dmj}. In \cite{oh:imrn},
it was enough to establish a non-constant component in the `limit' which
can be proved by a simple convergence argument under an energy bound.)

However such a \emph{dichotomy via the area} does not exist when $\phi_H^1(L) \to L$
in $C^0$-topoogy (or $H \to 0$ in $C^1$-topology). We now motivate possible
importance of such a study in relation to topological Hamiltonian dynamics
\'a la \cite{oh:hameo1}, \cite{oh:hameo2}.


\subsection{Topological Hamiltonian loops}
\label{subsec:top-flows}

In \cite{oh:hameo1}, M\"uller and the author introduced the notion of
hamiltonian topology on the space
$$
\CP^{ham}(Symp(M,\omega),id)
$$
of Hamiltonian flows $\lambda:[0,1] \to Symp(M,\omega)$ with
$\lambda(t) = \phi_H^t$ for some time-dependent Hamiltonian $H$.
We first recall the definition of this hamiltonian topology.

Following the notations of \cite{oh:hameo1}, we denote by $\phi_H$ the
Hamiltonian path
$$
\phi_H: t \mapsto \phi_H^t; \, [0,1] \to Ham(M,\omega)
$$
and by $\Dev(\lambda)$ the associated normalized Hamiltonian
\be\label{eq:Dev}
\Dev(\lambda) := \underline H, \quad \lambda = \phi_H
\ee
where $\underline H$ is defined by
\be\label{eq:underlineH}
\underline H(t,x) = H(t,x) - \frac{1}{\vol_\omega(M)} \int_M H(t,x)\, \omega^n.
\ee

\begin{defn}\label{defn:hamtopology} Let $(M,\omega)$ be a closed symplectic
manifold. Let $\lambda, \, \mu$ be smooth Hamiltonian paths.
The \emph{hamiltonian topology} of Hamiltonian paths is the metric topology induced by the metric
\be\label{eq:strong}
d_{ham}(\lambda,\mu): = \overline d(\lambda,\mu) +
\operatorname{leng}(\lambda^{-1}\mu).
\ee
\end{defn}

Now we recall the notion of topological Hamiltonian flows and
Hamiltonian homeomorphisms introduced in \cite{oh:hameo1}.

\begin{defn}[$L^{(1,\infty)}$ topological Hamiltonian flow]\label{defn:topflow} A continuous map
$\lambda: \R \to Homeo(M)$ is called a topological Hamiltonian flow
if there exists a sequence of smooth Hamiltonians $H_i: \R \times M
\to \R$ satisfying the following:
\begin{enumerate}
\item $\phi_{H_i} \to \lambda$ locally uniformly on $\R \times M$.
\item the sequence $H_i$ is Cauchy in the $L^{(1,\infty)}$-topology locally in time
and so has a limit $H_\infty$ lying in $L^{(1,\infty)}$ on any compact interval $[a,b]$.
\end{enumerate}
We call any such $\phi_{H_i}$ or $H_i$ an \emph{approximating sequence} of $\lambda$.
We call a continuous path $\lambda:[a,b] \to Homeo(M)$ a {\it
topological Hamiltonian path} if it satisfies the same conditions
with $\R$ replaced by $[a,b]$, and the limit $L^{(1,\infty)}$-function
$H_\infty$ called a \emph{$L^{(1,\infty)}$ topological Hamiltonian} or just
a \emph{topological Hamiltonian}.
\end{defn}

We call a topological Hamiltonian path $\lambda$ a loop if $\lambda(0) = \lambda(1)$.
Any approximating sequence $\phi_{H_i}$ of a topological Hamiltonian loop $\lambda$
has the property $\phi_{H_i}^1 \to id$ in addition to the properties
(1), (2) of Definition \ref{defn:topflow}.

\subsection{Thick-thin decomposition for $C^0$-small Lagrangian isotopy}

Motivated by the discussion laid out in the previous subsection, we are led to
analyze the behaviors of the Floer moduli space and of the Floer complex as
the boundary Lagrangian submanifold $\phi_{H_i}^1(L)$ for a sequence of
Hamiltonian diffeomorphisms $\phi_{H_i}^1 \to id$ in $C^0$-topology.

Unlike the case of $C^2$-small Hamiltonians, the dichotomy via the areas
described in the beginning of the present paper do not exist for the case of $C^1$-small
Hamiltonian $H$. The main purpose of the present paper is to generalize the construction
local Floer homology and its computation for the case where
the $C^2$-smallness of $H$ (or $C^1$-smallness of $\phi_H$) is replaced
by the the weaker hypothesis, the $C^0$-smallness of the time-one map
$\phi_H^1: t\mapsto \phi_H^t$ for any engulfable Hamiltonian path $\phi_H$.
There are two major differences between the cases of $C^1$-topology and of the $C^0$-topology
of Hamiltonian paths (or between the $C^2$-smallness of $H$ and $C^1$-smallness of $\phi_H$).
The first fundamental issue is that thin trajectories might not have small area but
could have large area for the
$C^0$-close Lagrangian submanifolds unlike the $C^1$-close case of Lagrangian submanifolds.
Because of this, instead of using the areas as in \cite{oh:imrn},
we will use the \emph{maximum principle} to single out `thin' trajectories
which turns out to be the best way of obtaining such decomposition even for the
$C^2$-small $H$'s in hindsight. However the thick-thin decomposition acquired
via the maximum principle does not differentiate the action filtration any more.
The second more technical issue is that
the $C^0$-topology is a priori too weak to uniformly control the analytical behavior of
pseudo-holomorphic curves with boundary lying on $\phi_H^1(L)$ in general
partially because we cannot establish uniform area bounds even for the thin
trajectories, while $C^1$-topology of Lagrangian boundary condition controls analysis of
pseudo-holomorphic curves.

To describe the problem in a precise manner, we need some digression.

Let $L \subset (M,\omega)$ be a compact Lagrangian submanifold and
let $V \subset \overline V \subset U$ be a pair of Darboux
neighborhoods of $L$. We denote $\omega = - d\Theta$ on $U$ where
$\Theta$ is the Liouville one-form on $U$ regarded as an open
neighborhood of the zero section of $T^*L$. Following \cite{oh:dmj},
\cite{spaeth}, we introduce the following notion. (Similar concept
was previously used by Laudenbach \cite{laud} in the context of
classical symplectic topology.)

We measure the size of the Darboux neighborhood $V$ by the following constant
\be\label{eq:dVTheta}
d(V,\Theta) := \max_{x \in V} |p(x)|, \quad x = (q(x),p(x)).
\ee
This constant is bounded away from $0$ and so
there exists some $\eta > 0$ depending only on
$(V,-d\Theta)$ (and so only on $(M,\omega)$) such that if $d_{C^0}(\phi_H^1,id) < \eta$, then
$\phi_H^1(L) \subset V$.

\begin{defn} We call an isotopy of Lagrangian submanifold $\{L_t\}_{0 \leq s \leq 1}$
of $L$ is called \emph{$V$-engulfable} if there exists a Darboux
neighborhood $V$ of $L$ such that $L_s \subset V$ for all $s$.
When we do not specify $V$, we just call the isotopy engulfable for $L$.

We call a (topological) Hamiltonian path $\phi_H$ engulfable if its graph
$\Graph \phi_H^t$ is engulfable in a Darboux neighborhood of the
diagonal $\Delta$ of $(M \times M, \omega \oplus -\omega)$.
\end{defn}

Following Weinstein's notation, we denote by $\frak{Iso}(L)$ the set of
Hamiltonian deformations of $L$. Define
\bea\label{eq:IsodL} \CH_{\delta}^{engulf}(L;V) & = & \{H
\mid \phi_H^t(L) \subset V\,  \forall t \in [0,1], \,
\overline d(\phi_H^1,id) \leq \delta\}\\
\frak{Iso}_\delta^{engulf}(L;V) & = & \{L' \in \frak{Iso}(L) \mid L' = \phi_H^1(L), \,
H \in \CH_{\delta}^{engulf}(L;V)\}.
\eea

One of the main goals of the present paper is to extend the notion of local Floer
homology introduced in \cite{floer:fixed,oh:imrn} for the $C^2$-small Hamiltonian $H$
to the case of $H$ such that
\begin{enumerate}
\item its Hamiltonian paths $\phi_H$ are $V$-engulfable,
\item its time-one map $\phi_H^1$ is $C^0$-small.
\end{enumerate}
Such a sequence of smooth Hamiltonian
paths naturally occurs as an approximating sequence of engulfable
topological Hamiltonian loop (based at the identity).

We would like to remark that it is established in \cite{oh:imrn} that if $\|H\|_{C^2} < C$
for sufficiently small $C> 0$, then
the following automatically hold:
\begin{enumerate}
\item its Hamiltonian paths $\phi_H$ is $V$-engulfable,
\item and the uniform area bounds of the associated
connecting Floer trajectories on $V$, where we regard $V$ as a neighborhood of
the zero section in the cotangent bundle so that we use the classical action
functional to measure the actions.
\item The path spaces $\CP(\phi_H^1(L),L)$ or $\CP(L,L)$ carry a distinguished
connected component on which the actions of any Hamiltonian chord become
uniformly small.
\end{enumerate}
In \cite{oh:imrn}, we mainly used the \emph{area} of Floer trajectories
to obtain the thick-thin decomposition of the Floer boundary operator
$\del = \del_0 + \del'$, \emph{which is equivalent to the corresponding dichotomy
in terms of filtration changes under the boundary map (or the Floer chain map)
for a $C^2$-small Hamiltonian $H$.}

However for the Hamiltonian $H$ of our interest in the present paper,
both properties (2) and (3) fail to uniformly hold even when we let $d_{C^0}(\phi_H^1,id) \to 0$.
Therefore there do neither exist a uniform gap in the filtration nor
uniform control of the filtration of the Floer complex
(or of the action bounds of the associated Hamiltonian chords).
This is a new phenomenon for the localization in the current topological Hamiltonian
context. Because of this lack of control of the filtration, we will instead use the more
geometric version of thick-thin decomposition mainly using the $C^0$
property of $\phi_H^1$ by exploiting the \emph{maximum principle.}

For this purpose, we fix a time-independent almost complex structure
$J_0$ that satisfies $J_0 \equiv J_g$ on $V$ where $J_g$ is the canonical
(Sasakian) almost complex structure on $V$ as a subset $T^*L$ which is induced by a Riemannian
metric $g$ on $L$, and suitably interpolated to outsider of $U$.
(We refer to \cite{floer:witten}, pp 321-323 \cite{oh:imrn} for the precise
description of $J_g$ and $J_0$ respectively.) We may assume $V$ has $J_0$-convex
boundary. We denote by
$$
\CJ_\omega(V,J_g)
$$
the set of such almost complex structures.

\subsection{Comparison of two Cauchy-Riemann equations}
\label{subsec:comparison}

For each given pair $(J,H)$, we consider the perturbed Cauchy-Riemann equation
\be\label{eq:CRHJ}
\begin{cases}
\frac{\del u}{\del \tau} + J\Big(\frac{\del u}{\del t}
- X_H(u)\Big) = 0\\
u(\tau,0) \, u(\tau,L_1) \in L
\end{cases}
\ee
which defines the Floer complex $CF_*(L,L;H)$ generated (over a suitable Novikov ring) by the set
$\CC hord(H;L,L)$ defined by
\be\label{eq:Chord}
\CC hord(H;L,L)): = \{z:[0,1] \to M \mid \dot z = X_H(t,z), \, z(0), \, z(1) \in L\}.
\ee
We call any such element $z$ in $\CC hord(H;L,L))$ a Hamiltonian chord
of $L$. This Cauchy-Riemann equation is called the \emph{dynamical version} in \cite{oh:jdg}.

Equivalently one can also consider the \emph{genuine} Cauchy-Riemann equation
\be\label{eq:dvtildeJ}
\begin{cases}\dvdtau + J^H \dvdt = 0 \\
v(\tau ,0) \in \phi_H^1(L), \,
v(\tau ,1) \in L
\end{cases}
\ee
for the path $u:\R \to \CP(\phi_H^1(L),L)$
$$
\CP(\phi_H^1(L),L) = \{\gamma: [0,1] \to T^*N \mid \gamma (0) \in \phi_H^1(L), \,  \gamma(1) \in L\}
$$
and $J^H_t = (\phi_H^t(\phi_H^1)^{-1})_*J_t$. We note that whenever $\supp \phi_H \subset V$,
$J^H_t \in \CJ_\omega(V,J_g)$ for all $t \in [0,1]$.
We call this version the \emph{geometric version}.
The upshot is that there is a filtration preserving isomorphisms between the dynamical version and
the geometric version of the Lagrangian Floer theories.

We now describe the geometric version of the Floer homology in some more details.
We denote by $\widetilde \CM(L_H,L;J^H)$ the set
of finite energy solutions and $\CM(L_H,L;J^H)$ to be its quotient by $\R$-translations.
In the unobstructed case \cite{fooo:book}, this gives rise to the geometric
version of the Floer homology $HF_*(\phi_H^1(L),L, \widetilde J)$ of the type
\cite{floer:Morse} whose generators are the intersection
points of $\phi_H^1(L)\cap L$. An advantage of this version is that it depends only
on the Lagrangian submanifold $(\phi_H^1(L),L)$ depending loosely on $H$.

The following is a straightforward to check but is a crucial lemma.

\begin{lem}\label{lem:equiv}
\begin{enumerate}
\item The map $\Phi_H: \phi_H^1(L) \cap L \to \CC hord(H;L,L)$ defined by
$$
x \mapsto z_x^H(t)= \phi_H^t\left(\phi_H^{-1}(x)\right)
$$
gives rise to the one-one correspondence between the set
$\phi_H^1(L) \cap L \subset \CP(\phi_H^1(L), L)$ as constant paths and the set of
solutions of Hamilton's equation of $H$.
\item The map $a \mapsto \Phi_H(a)$ also defines a one-one
correspondence from the set of solutions of \eqref{eq:dvtildeJ} and that of
\be\label{eq:dvJH}
\begin{cases}
\dvdtau + J^H \dvdt = 0 \\
v(\tau ,0) \in \phi_H^1(L),  \,  v(\tau ,1) \in L
\end{cases}
\ee
where $J^H = \{ J^H_t \} , J^H_t: = (\phi^t_H (\phi^1_H)^{-1})^* J_t $.
Furthermore, \eqref{eq:dvJH} is regular if and only if \eqref{eq:CRHJ} is regular.
\end{enumerate}
\end{lem}

Once we have transformed \eqref{eq:CRHJ} to \eqref{eq:dvJH}, we can further
deform $J^H$ to the constant family $J_0$ inside $\CJ_\omega(V,J_g)$ and consider
\be\label{eq:CRdvJ0}
\begin{cases}
\dvdtau + J_0 \dvdt = 0 \\
v(\tau ,0) \in \phi_H^1(L),  \, v(\tau ,1) \in L
\end{cases}
\ee
for each given $J_0\in \CJ_\omega$, a time-independent family. We will
fix a generic $J_0$ in the rest of the paper and assume $L$ is transversal to
$\phi_H^1(L)$ by considering a $C^\infty$-small perturbation of $H$ if necessary.
This latter deformation preserves the filtration of
the associated Floer complexes \cite{oh:jdg}. A big advantage of
considering this equation is that it enables us to study the behavior of
spectral invariants for a sequence of $L_i$ converging to $L$
\emph{in Hausdorff distance}.

The following thick-thin decomposition of the Floer moduli spaces  is a crucial ingredient. This is a variation of
Proposition 4.1 \cite{oh:imrn} in the $C^0$ context.

\begin{thm}[Compare with Proposition 4.1 \cite{oh:imrn}]\label{thm:Sigma-d}
Let $L \subset (M,\omega)$ be a compact
Lagrangian submanifold and let $V \subset \overline V \subset U$ be
a pair of Darboux neighborhoods of $L$. Consider a $V$-engulfable
Hamiltonian path $\phi_H$. Then whenever
$\overline d(\phi_H^1,id) \leq \delta$ for any $\delta <  d(V,\Theta)$,
any solution of $v$ of \eqref{eq:CRdvJ0} satisfies one of the following alternatives:
\begin{enumerate}
\item Either
\be\label{eq:verythin}
\Image v \subset D_\delta(L) \subset V
\ee
where $D_\delta(L)$ is the $\delta$-neighborhood of $L$.
\item or $\Image v \not \subset V$. In this case, we also have
$\int v^*\omega \geq C(J_0,V)$ where $C(J_0,V) > 0$ is a constant
depending only on $\delta$ and $V$.
\end{enumerate}
\end{thm}

We call $v$ a \emph{thin} trajectory if $\Image v \subset V$ and a \emph{thick} trajectory
otherwise. We call a thin trajectory \emph{very thin}
if it satisfies \eqref{eq:verythin} in addition. This theorem basically
says that all thin trajectories are indeed very thin and all thick trajectories
have area bounded below away from zero.
The proof of this theorem is an easy application of maximum principle
on the $J_0$-convex domain $V$ and the monotonicity formula for the $J_0$-holomorphic curves.
We would like to emphasize that the meaning of thin trajectories here is
different from that of \cite{oh:imrn} (or the `short' trajectories in \cite{chekanov:newton}) in that
\emph{ they could have large areas} unlike the case of latters.

This theorem enables us to define the local Floer homology
in a well-defined way by counting thin trajectories.
We denote this local Floer homology by
$$
HF_*^{[id]}(\phi_H^1(L),L;U), \quad \mbox{
or }\, HF_*^{[id]}(H,(L,L);U).
$$
By definition, $HF^{[id]}(\phi_H^1(L),L;U)$ is always
well-defined \emph{without} any unobstructedness assumption of $L \subset M$
such as exactness or monotonicity of the pair $(L,M)$ or the unobstructedness in the sense of
\cite{fooo:book}.

Once the above thick-thin decomposition results of the Floer moduli
spaces for the boundary and for the chain maps are established,
essentially the same isolatedness argument as in \cite{oh:imrn}
gives rise to the following computation.

\begin{thm}\label{thm:local} Let $L\subset M$ be as above and  $U$
be a Darboux neighborhood of $L$ and $\CH: s \mapsto H(s)$ a family
of $U$-engulfable Hamiltonians with $H(0) = 0$. Then if $\max_{s \in
[0,1]}\overline d(\phi_{H(s)}^1,id) < \delta$ and $|J_t-J_0|_{C^1}<
\delta$ for some time independent $J_0$ and if $J$ is
$(L,\phi_H^1(L))$-regular, then
$$
HF_*(H,L;J;U) \cong H_*(L;\Z).
$$
\end{thm}

We would like to emphasize that the presence of the engulfable homotopy
$\CH$ is crucial in the statement of this theorem, because the commonly used
the linear homotopy $s \mapsto s\, H$ may not be
$U$-engulable and so may not induce a chain map between
the local Floer complex, even when $H_0,\, H_1$ are $U$-engulfable. In this regard,
statement and the proof of this theorem given in section \ref{sec:local-compu} may be the
most novel points of the mathematics of the present article.

To perform the above computation, we need to study the
behavior of the local Floer homology under the change of Hamiltonians. In this regard,
we consider a 1-parameter family of Hamiltonians (or a
2-parameter family of functions on $M$) $\CH =\{H(s)\}_{0 \leq s
\leq 1}$ with $H(0) \equiv 0$ and
\be\label{eq:para<d} \max_{s \in
[0,1]}\overline d(\phi_{H(s)}^1,id) < \delta
\ee
for a sufficiently
small $\delta = \delta_0(M,\omega;J_0)$. We fix an elongation function
$\rho:\R \to [0,1]$ satisfying
\bea\label{eq:rho}
\rho(\tau) & = & \begin{cases} 0 \quad & \tau \leq 0 \\
1 \quad & \tau \geq 1
\end{cases}\nonumber \\
\rho' & \geq & 0
\eea
and define its dual $\widetilde \rho : = 1-\rho$.
We will consider the lemma in the Lagrangian setting over the path
$s \mapsto H(s)$ for $\CH = \{H(s)\}_{s \in [0,1]} \subset \CH^{engulf}_\delta(M)$
with $H(0) \equiv 0$ for $\delta$ sufficiently small.
Again the smallness will depend only on $(M,\omega)$.

We consider the Cauchy-Riemann equation with moving boundary condition
\be\label{eq:CRdvJ0moving}
\begin{cases} \dvdtau + J_0 \dvdt = 0 \\
v(\tau,0) \in \phi_{H(\rho(\tau))}^1(L), \, v(\tau,1) \in L.
\end{cases}
\ee
Then we prove the following analog to Theorem \ref{thm:Sigma-d} for chain maps.
This is the analogue of the \emph{handle sliding lemma} from \cite{oh:ajm1,oh:minimax}
in which was studied the case with $C^2$-smallness of Hamiltonians replaced by the smallness in hamiltonian topology
(and also in the Lagrangian context).

\begin{thm}[Handle sliding lemma]\label{thm:handle}
Consider the path $\CH: s \mapsto H(s)$ of engulfable Hamiltonians
$H(s)$ satisfying \eqref{eq:para<d} and fix an elongation
function $\rho:\R \to [0,1]$. Then whenever
$\overline d(\phi_{H(s)}^1,id) \leq \delta < d(V,\Theta)$, any finite energy solution $v$
of \eqref{eq:CRdvJ0moving} satisfies one of the following alternatives:
\begin{enumerate}
\item Either
\be\label{eq:verythin2}
\Image v \subset D_\delta(L) \subset V,
\ee
\item or $\Image v \not \subset V$.
In this case, we also have
$\int v^*\omega \geq C(J_0,V)$ where $C(J_0,V) > 0$ is a constant
depending only on $\delta$ and $V$.
\end{enumerate}
\end{thm}

Similarly as before done for the boundary map,
we can transform  a solution $v$ of \eqref{eq:CRdvJ0} to that of
the perturbed Cauchy-Riemann equation with fixed boundary condition
\be\label{eq:CRKH}
\begin{cases} \frac{\del u}{\del \tau} - X_{K(\rho(\tau))}(u) +
J\Big(\frac{\del u}{\del t} - X_{H(\rho(\tau))}(u)\Big) = 0\\
\lim_{\tau \to -\infty}u(\tau) = z^-,  \lim_{\tau \to
\infty}u(\tau) = z^+.
\end{cases}
\ee
The Floer chain map $h_{H^\rho}: CF_*(H^0) \to CF_*(H^1)$ can be
defined by considering either the suitable moduli space of solutions \eqref{eq:CRKH}
or that of \eqref{eq:CRdvJ0moving}.

\subsection{Statement of main results}

Using the above constructed local Floer homology, we can assign the local spectral invariants
which we denote by $\rho^{lag}_U(H;a)$ for $a \in H^*(L;\Z)$. We will
restrict to the case $PD[M] = 1$. To highlight the localness of the invariant
we denote $\rho^{lag}_U(H;1_0)$ the corresponding invariant. Denote the
global spectral invariant associated to $1$ by $\rho^{lag}(H;1)$.

%

By specializing to the case of zero section $o_N$ of $T^*N$,
 we can define the local Floer complex
$$
(CF_*(F;U,T^*N),\del_U)
$$
for any $F \in \CH_\delta^{engulf}(T^*N)$ provided $\delta > 0$ is sufficiently small.

When Theorem \ref{thm:Sigma-d} and Theorem \ref{thm:handle} are
applied to the cotangent bundle $T^*L$, we obtain the following

\begin{cor}\label{cor:exact} Consider a pair of open neighborhoods $V \subset \overline V \subset U$ of
$o_L$ in $T^*L$ be given where $V$ is $J_0$-convex. Assume $\CH = \{H(s)\}$ is
an engulfable isotopy with $F = H(1)$ satisfying \eqref{eq:para<d}.
Fix an elongation function $\rho:\R \to [0,1]$ and consider the equation
\eqref{eq:CRdvJ0moving}.
Then there exists $\delta > 0$ with $\delta < d(V,\theta)$ such that whenever
$$
\max_{s \in [0,1]} \overline d(\phi_{H(s)}^1,id) < d(V,\Theta),
$$
the followings hold:
\begin{enumerate}
\item For $F=H(1)$,
any solution of $v$ of \eqref{eq:CRdvJ0} with $\Im v \subset U$ is  very thin.
\item
Fix an elongation function $\rho:\R \to [0,1]$ and consider the equation
\eqref{eq:CRdvJ0moving}.
Then any finite energy solution $v$ with $\Im v \subset U$
is thin (and so very thin).
\end{enumerate}
\end{cor}

When we specialize our construction, using Corollary \ref{cor:exact}, to a $J_0$-convex neighborhood of
the zero section in the cotangent bundle, we can define the local version $\rho^{lag}_V(F;1_0)$
of the Lagrangian spectral invariants. We refer to section \ref{sec:cotangent} for the detailed
construction of this local invariant.

Once we have achieved localizations of various entities arising in
Floer complex in the previous subsection, the following equality can be proven
by the same argument used in the proof of Theorem \ref{thm:localrho=globalrho}
using the localized version of Lagrangian spectral invariants and
basic phase function.

\begin{thm}\label{thm:rhoVF<rhoF} Fix an open neighborhood $V \subset T^*L$ of
$o_L \subset T^*L$ that is $J_0$-convex. Let $\CH=\{H(s)\}$ be an engulfable isotopy
with $H(0) = 0$ and $H(1) =F$. Then for any $F \in \CH^{engulf}_\delta(M;V)$,
$$
\rho^{lag}_V(F;1_0) = \rho^{lag}(F;1).
$$
\end{thm}

Denote by $f_F^V$ the basic phase function defined by
$$
f_F^V(q) = \rho_V^{lag}(F;\{q\})
$$
where $\rho_V^{lag}(F;\{q\})$ is the spectral invariant defined by considering the
local Floer complex $CF(L, T_q^*N;V)$ instead of the global complex $CF(L,T_q^*N)$ on
$T^*N$.

\begin{thm}\label{thm:localfH=globalfH} Let $V \subset T^*N$ be as before. Then
$$
f_F^V = f_F
$$
for any $V$-engulfable $F$.
\end{thm}
One important consequence of the above theorems is the inequalities
\be\label{eq:lagV<E-F}
\rho^{lag}_V(F;1_0), \quad \max f_H^V \leq E^-(F).
\ee
\begin{rem} We recall from \cite{oh:jdg,oh:alan} that the proof of the inequality
$\rho^{lag}(F;a) \leq E^-(F)$ therein is based on the computation of the action
changes under the linear homotopy $s \mapsto s\, F$. This cannot be directly applied to
the study of the local version of spectral invariants because the linear homotopy
may not be $V$-engulfable as mentioned before.
\end{rem}

In Appendix, we also localize the Hamiltonian Floer complex for a future purpose.
We apply the above constructions to the graph
$$
\Graph \phi_F^1 = \{(\phi_H^1(x),x) \mid x \in M\} \subset M \times M
$$
of engulfable Hamiltonian $F$ on $M$ satisfying
\be\label{eq:condition-F}
\overline d(\phi_F^1,id) < \delta
\ee
for a sufficiently small $\delta > 0$. We define
$$
\CH_{\delta}^{engulf}(M) \subset C^\infty([0,1] \times M,\R)
$$
to be the set of such Hamiltonian $F$'s, and
call the associated Hamiltonian path an engulfable Hamiltonian \emph{$C^0$-approximate loop}.

Consider $\CU \subset \CL_0(M)$ defined by
$$
\CU = \CU(U_\Delta) =\{ \gamma \in \CL_0(M) \mid (\gamma(t),\gamma(0)) \in U_\Delta\}.
$$
We define the local Floer homology
$$
HF_*^{[id]}(F,J;\CU), \quad \CU \subset \CL_0(M)
$$
by counting the `thin' trajectories such that
their images are contained in a neighborhood
$\CU$ of the set of constant paths in $M$. Again we would like to
emphasize that we have not control of the area or filtrations.

We denote by $\rho^{ham}_\CU(\phi_F;1) = \rho^{ham}_\CU(\underline F;1)$ the (local) spectral
invariant associated to $1 \in H^*(L)$.

From now on, we will always assume that all the Hamiltonians in the rest of the paper are engulfable
one way or the other, unless otherwise said explicitly.

We would like to thank D. McDuff and H. Hofer for pointing out a
crucial gap in our $W^{1,p}$-precompactness proof in the previous version
of the present paper. This forces us to abandon the area argument to obtain the thick-thin
decomposition of Floer moduli spaces but exploit the maximum principle instead to obtain
a decomposition result that will do our purpose of extracting the local Floer
complex our of the global Floer complex. We also thank an anonymous referee for
many suggestions to improve presentation of the paper.

\section{Local Floer chain module of engulfable Hamiltonian path $\phi_H$}
\label{sec:local}

From now on, we will fix a pair of
Darboux neighborhood $V \subset \overline V \subset U$ of $L$ in $M$ and
assume $H$ is $V$-engulfable, i.e., satisfies
\be\label{eq:isolated2}
\phi_{H^t}(L) \subset V \subset \overline V \subset U
\ee
for all $t \in [0,1]$.

Next we recall the Lagrangian analogue of
the Novikov ring $\Gamma_\omega = \Gamma(M,\omega)$ from \cite{fooo:book}.
Denote by $I_\omega:\pi_2(M,L) \to \R$
the evaluations of symplectic area. We also define another integer-valued
homomorphism $I_\mu:\pi_2(M,L) \to \Z$ by
$$
I_\mu(\beta) = \mu\left(w^*TM,(\del w)^*TL\right)
$$
which is the Maslov index of the bundle pair $(w^*TM,(\del w)^*TL)$
for a (and so any) representative $w:(D^2,\del D^2) \to (M,L)$ of $\beta$.

\begin{defn}\label{defn:GLomega}
We define
$$
\Gamma_{(\omega,L)} = \frac{\pi_2(M,L)}{\ker I_\omega \cap \ker I_\mu}.
$$
and $\Lambda(\omega,L)$ to be the associated Novikov ring.
\end{defn}

We briefly recall the basic properties on the
Novikov ring $\Lambda_{(\omega,L)}(R)$ where $R$ is a commutative ring
where $R$ could be $\Z_2,\, \Z$ or $\Q$ for example. We will just use the letter $R$
for the coefficient ring which we do not specify. Basically $R$ will be $\Q$
when the associated moduli space is orientable as in the case of $\Graph \phi_H^1$
for a Hamiltonian diffeomorphism $\phi_H^1$ which is of our main interest.

We put
$$
q^\beta = T^{\omega(\beta)} e^{\mu_L(\beta)},
$$
and
$$
\operatorname{deg}(q^\beta) = \mu_L(\beta), \quad E(q^\beta) = \omega(\beta)
$$
which makes $\Lambda_{(\omega,L)}$ and $\Lambda_{0,(\omega,L)}$ become a graded
ring in general. We have the canonical valuation
$\nu: \Lambda_{(\omega,L)} \to \R$ defined by
$$
\nu\left(\sum_\beta a_\beta T^{\omega(\beta)}e^{\mu_L(\beta)}\right) =\min\{ \omega(\beta) \mid a_\beta \neq 0 \}
$$
It induces a valuation on $\Lambda_{(\omega,L)}$
which induces a natural filtration on it. This makes $\Lambda_{(\omega,L)}$ a filtered graded ring.
For a general Lagrangian submanifold, this ring may not even be Noetherian but it is so if
$L$ is rational, i.e., $\Gamma(L;\omega)$ is discrete.

Now consider a nondegenerate $V$-engulfable Hamiltonian $H$ where $V$ is
a given Darboux neighborhood of $L$. We denote by $\Omega(L,L)$ the set of
paths $\gamma:[0,1]$ with $\gamma(0), \, \gamma(1) \in L$. In general $\Omega(L,L)$
is not connected but it has the distinguished component of constant paths, which
we denote by
$$
\Omega_0(L,L).
$$
When $H$ is $V$-engulfable, the path space $\Omega(\phi_H^1(L),L)$ also carries the
distinguished component of the path $t \mapsto \phi_H^t(p)$ for $p \in L$. We denote by
$$
\Omega_0(\phi_H^1(L),L)
$$
the corresponding component. Then we denote by $\widetilde \Omega_0(L,L)$ the
Novikov covering space and $\pi: \widetilde \Omega_0(L,L) \to \Omega_0(L,L)$ the
projection. We denote by $[z,w]$ an element of $\widetilde \Omega_0(L,L)$.

Following \cite{chekanov:area} we say that two elements of $\Crit \CA_H$ are equivalent
if they belong to the same connected component of the set
$$
\pi^{-1}\left(\{\gamma \in \Omega_0(L,L) \mid \gamma([0,1]) \subset U\}\right) \subset \widetilde \Omega_0(L,L).
$$
Then the projection $\pi: \widetilde \Omega_0(L,L) \to \Omega_0(L,L)$ bijectively maps each equivalence class of
$\Crit \CA_H$ to $\CC hord(H;L,L)$. In the current case of $V$-engulfable Hamiltonian,
 there is a `canonical equivalence class' represented by the pairs $[z,w_z]$
for each given chord $z \in \CC hord(H;L,L)$, where $w_z$ is the (homotopically) unique
cone-contraction of $z$ to a point in $L$.
We denote this equivalence class by $\Crit^{[id]}\CA_H \subset \Crit \CA_H$. This
provides a canonical section of $\pi: \widetilde \Omega_0(L,L) \to \Omega_0(L,L)$
when restricted to $\CC hord(H;L,L) \subset \Omega_0(L,L)$.
This in turn induces a natural
$\Gamma_{(\omega,L)}$-action on $\Crit \CA_H$ which gives rise to the bijection
$$
\Crit^{[id]} \CA_H \times \Gamma_{(\omega,L)} \to \Crit \CA_H.
$$
\begin{rem}\label{rem:cut-offing}
Note that for any $[z,w] \in \Crit^{[id]} \CA_H$, \eqref{eq:isolated2} implies
$z(t) \in V$ since $z(0) \in \phi_H^1(L)$. Therefore the action value
$\CA_H([z,w])$ will not change even if we cut-off $H$ outside $V$.
\end{rem}

We denote
$$
\Crit^{[g]} \CA_H = g \cdot \Crit^{[id]} \CA_H, \quad g \in \Gamma_{(\omega,L)}.
$$
Then we denote their associated $R$-module by
$$
CF_*^{[g]}((L,L),H;U), \quad CF_*^{[id]}((L,L),H;U) = CF_*^{[id]}((L,L),H;U).
$$
We want to remark that $CF_*^{[id]}((L,L),H;U)$ coincides with the local Floer complex that was used
by the author in \cite{oh:imrn} for the case of $C^2$-small cases.

The above discussion in turn gives rise to the isomorphism
$$
CF^{[g]}((L,L),H;U) \otimes_R \Lambda_{(\omega,L)}
\cong CF_*((L,L);H)
$$
as a $\Lambda_{(\omega,L)}$-module for each $g \in \Lambda_{(\omega,L)}$.
Following \cite{chekanov:newton,chekanov:area}, we denote
\be\label{eq:lengu}
\leng(u) : = E_J(u) = E_{J_0}(v) = \area(v).
\ee
Now we note that the Floer (pre)-boundary map
$$
\del:CF_*((L,L);H) \to CF_*((L,L);H)
$$
is $\Lambda_{(\omega,L)}$-equivariant and has the decomposition
\be\label{eq:dellambda}
\del = \sum_{\lambda \in \R_{\geq 0}} \del_\lambda
\ee
where $\del_\lambda$ is the contribution arising from $u \in \CM(L,L;H)$ with
$$
\leng(u) = \lambda > 0.
$$

\section{Thick-thin dichotomy of Floer trajectories}
\label{sec:thick-thin}

This section is a modification of section 3 of \cite{oh:imrn} which
treats the case of $C^2$-small perturbation of Hamiltonians $H$.
In this section, we will replace the condition of $\phi_H$ being
$C^1$-small by that of $\phi_H$ being $C^0$-small.

Consider a sequence $v: \R \times [0,1] \to M$ of solutions of
\eqref{eq:CRdvJ0} associated to $H$ and $J_0$. We re-state Theorem \ref{thm:Sigma-d}
here.

\begin{thm}\label{thm:thick-thin2}
Let $L \subset (M,\omega)$ be a compact
Lagrangian submanifold and let $V \subset \overline V \subset U$ be
a pair of Darboux neighborhoods of $L$. Consider a $V$-engulfable
Hamiltonian path $\phi_H$.
Then there exists $\delta > 0$ depending only on
$\e$ (and $(M,\omega)$) such that whenever
$\overline d(\phi_H^1,id) \leq \delta$,
any solution of $v$ of \eqref{eq:CRdvJ0} satisfies one of the following alternatives:
\begin{enumerate}
\item $\Image v \subset V$ and $\max d(v(z), o_L) \leq d_{\text{\rm H}}(\phi_H^1(L),L)$,
\item $\Image v \not \subset V$ and
$\int v^*\omega \geq C(J_0,V)$ where $C(J_0, V) > 0$ is a constant
depending only on $V$.
\end{enumerate}
\end{thm}
\begin{proof} Suppose that $\Image v \not \subset V$.
Then there exists a point $v(z) \not\in V$ and so
$$
d(v(z), v(\del(\R \times [0,1])) \geq \min\{d_{\text{\rm H}}(\del V,\phi^1(L)),
d_{\text{\rm H}}(\del V,o_L)\}.
$$
Then monotonicity formula implies
$$
\int v^*\omega \geq C' \cdot \left(\min\{d_{\text{\rm H}}(\del V,\phi^1(L)),
d_{\text{\rm H}}(\del V,o_L)\}\right)^2
$$
where $C'$ is the monotonicity constant of $(M,\omega,J_0)$ in the monotonicity formula.
Considering $\delta < \frac{1}{4}\cdot d_{\text{\rm H}}(\del V,o_L)$ and setting
$$
C(J_0,V) : = \frac{1}{2}C' \cdot (d_{\text{\rm H}}(\del V,o_L) - \delta)^2
\geq \frac{1}{4}C' d_{\text{\rm H}}(\del V,o_L)^2
$$
(2) follows.

For the curve $v$ of the type (1), the maximum principle applied to
$J_0$-holomorphic curves contained in the Darboux neighborhood, we
obtain the maximum distance of $v(z)$ from $L$ is achieved on
the boundary $\R \times \{0,1\}$. But by the boundary condition,
we have
$$
\max_{z \in \R \times \{0,1\}} d(v(z),o_L) \leq d_{\text{\rm H}}(\phi_H^1(L),L)).
$$
This finishes the proof.
\end{proof}

\begin{rem} We would like to note that the area property for the trajectories $v$
of the type (2) spelled out as
$$
\int v^*\omega \geq C(J_0,V)
$$
with constant $C(J_0, V) > 0$ independent of $H$ will not be used in this paper.
\end{rem}

We now decompose $\del$ into
\be\label{eq:del}
\del = \del_{(0)} + \del'
\ee
where $\del_{(0)}$ is the sum of contribution of thin trajectories and $\del'$
that of thick trajectories.

\begin{rem}\label{rem:two-decompositions}
We would like to emphasize that \emph{even when $H$ is
$C^1$-small} this decomposition does not respect the one given in \eqref{eq:dellambda}.
This is a contrast from the case of $C^2$-small $H$: in that case it was proven in \cite{oh:imrn}
that there is a constant $\delta_0(H)$ satisfying $\delta(H) \to 0$ as $\|H\|_{C^2} \to 0$ such that
all thin trajectories have area less than $\delta_0(H)$ and that all thick trajectories have area
greater than $\frac{1}{2} A(M,L,J_0)$ and hence
$$
\del_{(0)} = \sum_{|\lambda| < \delta_0(H)} \del_{\lambda}, \quad
\del' = \sum_{\lambda > \frac{1}{2}  A(M,L,J_0)} \del_\lambda.
$$
\end{rem}

We denote $u \in \supp \del, \, \supp \del_{(0)}$, and $\supp \del'$ respectively,
if the map $u$ nontrivially contributes to the corresponding operators.

\begin{defn} We call $(CF_*^{[id]}((L,L),H;U), \del_U)$ the \emph{local Floer complex}
of $H$ in $U$ which is defined to be
\beastar
CF_*^{[id]} ((L,L),H;U) & = & R \cdot\{\Crit^{[id]} \CA_H\}, \\
\del_U & = & \del_{(0)}\Big|_{CF_*^{[id]} ((L,L),H;U)}.
\eeastar
\end{defn}
The $\Lambda_{(\omega,L)}$-equivariance of $\del$ gives rise to
$$
\widehat g \circ \del_{(0)}|_{CF_*^{[id]} ((L,L),H;U)} = \del_{(0)}|_{CF_*^{[g]} ((L,L),H;U)} \circ \widehat g
$$
and $\widehat g$ carries a natural weight given by
$$
\CA_F(g \cdot[z,w]) - \CA_F([z,w]), \, [z,w] \in \Crit \CA_F
$$
which does not depend on the choice of $ [z,w] \in \Crit \CA_F$.
In fact this real weight is nothing but the value $\omega([g])$.

\begin{prop}\label{prop:HF=H} Let $\delta > 0$ where $\delta$ is
the constant given in Theorem \ref{thm:thick-thin2}. Then $\del_U^2 = 0$ and so the local Floer homology
$$
HF_*^{[id]}((L,L),H;U) = \ker \del_U/ \operatorname{im} \del_U
$$
is well-defined.
\end{prop}
\begin{proof}
Since all the thin trajectories have their image contained in
the Darboux neighborhood $U$, concatenations of thin trajectories also thin and
the thin part of Floer moduli spaces for the pair $(\phi_H^1(L),L)$
does not bubble-off. Then the standard compactness and gluing argument immediately
finishes the proof.
\end{proof}

In the next section we will compute the group $HF_*^{[id]}((L,L),F;U)$, when
$F = H(1)$ for a 2-parameter family $\CH = \{H(s)\}_{s \in [0,1]}$ with $H(0) = 0$
and $H(s) \in \CH^{engulf}_\delta(M)$.
We denote by
$$
\overline d(\phi_\CH^1,id): = \max_{s \in [0,1]} \overline d(\phi_{H(s)}^1,id)
$$
the $C^0$-distance of $\CH$ to the constant family $id$.

\section{Handle sliding lemma for engulfable isotopy of $C^0$-approximate loops}
\label{sec:handle}

In this section, we examine another important element in the chain
level theory, the {\it handle sliding lemma} introduced in
\cite{oh:ajm1} for the Hamiltonian $H$ that is sufficiently $C^2$-small.
We will consider the lemma in the Lagrangian setting over the path
$s \mapsto H(s)$ for $\CH = \{H(s)\}_{s \in [0,1]} \subset \CH^{engulf}_\delta(M)$
with $H(0) \equiv 0$ for $\delta$ sufficiently small.
Again the smallness will depend only on $(M,\omega)$.

For a family $\CH = \{H(s)\}_{s \in [0,1]}$, we also study the comparison
of this equation with the moving boundary condition.
For such a family, we consider the geometric version first
\be\label{eq:dvtildeJ}
\begin{cases}\dvdtau + J_0 \dvdt = 0 \\
v(\tau ,0) \in \phi_{H(\rho(\tau))}^1(L), \,
v(\tau ,1) \in L
\end{cases}
\ee
for the path $v:\R \times [0,1] \to M$. If we define a map
$u: \R \times [0,1] \to M$
$$
u(\tau,t) = \phi_{H(\rho(\tau))}^t(\phi_{H(\rho(\tau))}^1)^{-1}(v(\tau,t)),
$$
A simple calculation proves that $u$ satisfies $u(\tau,0), \, u(\tau,1) \in L$ and
\be\label{eq:CRKJH}
\begin{cases}
\dudtau - X_{K(\rho(\tau))}(u) + J\left(\dudt - X_{H(\rho(\tau))}(u)\right) = 0\\
u(\tau,0), \, u(\tau,1) \in L
\end{cases}
\ee
where $K$ is the $s$-Hamiltonian generating the Hamiltonian vector field
$$
X_K(s,t,x) : = \frac{\del \phi}{\del s}(\phi^{-1}(s,t,x))
$$
of the 2-parameter family
$(s,t) \mapsto \phi(s,t) = \phi_{H(s)}^t \phi_{H(s)}^{-1}$ and $J = J(s,t) = (\phi(s,t))_*J_0$.
We would like to highlight the presence of the terms $X_{K(\rho(\tau))}(u)$ in
the above equation for $u$ and the definition of energy of $u$.
The associated \emph{off-shell} energy of \eqref{eq:CRKJH} is given by
\be\label{eq:intv*omega}
E_{(H,K),J;\rho}(u) =\frac{1}{2} \int_{-\infty}^\infty\int_0^1
\left|\dudtau - X_{K(\rho(\tau))}(u)\right|_J + \left|\dudt - X_{H(\rho(\tau))}(u)\right|_J^2 \, dt\, d\tau.
\ee
which coincides with
$$
\int_{-\infty}^\infty\int_0^1 \left|\dudt - X_{H(\rho(\tau))}(u)\right|_J^2 \, dt\, d\tau
$$
\emph{on shell}. The proof of the on-shell identities
$$
\int v^*\omega  = E_{J_0}(v) = E_{(H,K),J;\rho}(u)
$$
is straightforward and so omitted.
With these correspondences, we have the obvious analog to Lemma \ref{lem:equiv}
for the moving boundary condition, whose precise statement we omit.

Here we re-state Theorem \ref{thm:handle} and give its proof here.

\begin{thm}[Handle sliding lemma]\label{thm:handle2}
Consider the path $\CH: s \mapsto H(s)$ of engulfable Hamiltonians
$H(s)$ satisfying \eqref{eq:para<d} and fix an elongation
function $\rho:\R \to [0,1]$.  Then there exists $\delta > 0$ such that if
$\overline d(\phi_{H(s)}^1,id) < \delta < d(V,\Theta)$, any finite energy solution $v$
of \eqref{eq:CRdvJ0moving} satisfies one of the following alternatives:
\begin{enumerate}
\item if $\Image v \subset V$, $\max_z d(v(z),o_L) \leq \overline d(\phi_{H(s)}^1,id)$,
\item if $\Image v \not \subset V$,
$\int v^*\omega \geq C(J_0,V)$ where $C(J_0,V) > 0$ is a constant
depending only on $J_0$ and $V$.
\end{enumerate}

\end{thm}
\begin{proof} The proof is the same as that of Theorem \ref{thm:thick-thin2} and
so omitted.
\end{proof}

Now Theorem \ref{thm:handle2} together with this dichotomy of thick-thin
trajectories
enable us to decompose the Floer-Piunikhin (pre)-chain map
$$
\Psi_\CH: C_*(L) \to CF_*((L,L),H(1);U)
$$
into the thick-thin decomposition
\be\label{eq:Psi}
\Psi_\CH = \Psi_{\CH,(0)} + \Psi_\CH'
\ee
similar to \eqref{eq:del}. Again it follows from Theorem
\ref{thm:handle} that those $v$'s contributing non-trivially to
$\Psi_{\CH,(0)}$ are very thin (and those contributing to $\Psi_\CH'$
has area bigger than $C(J_0,V)$.)

We refer to section 5.3 \cite{fooo:book} or section 5 \cite{fooo:disc} for
the details of the construction of the Floer-Piunikhin (pre)-chain map $\Psi_\CH$.

\begin{rem}\label{rem:nonzerof}
The above Handle sliding lemma can be also proved by the same
argument for the Floer chain map between $f$ and $H(1) \# f$ when
$|f|_{C^2}$ is sufficiently small relative to $C(V,J_0)$.
This way one can avoid using the Bott-Morse version of Floer chain
map, the Floer-Piunikhin (pre)-chain map $\Psi_\CH$.
\end{rem}

\section{Computation of Local Floer homology $HF_*^{[id]}((L,L),H;U)$}
\label{sec:local-compu}

The role of the $C^2$-smallness in the construction of local Floer complex
$$
HF_*^{[id]}((L,L),H;U)
$$
in \cite{oh:imrn} was two-fold.
One is to make its flow $\phi_H$ $C^1$-small which
gives rise to a thick-thin decomposition of Floer complex. The other is
for the construction of (local) chain isomorphism between the
singular complex of $L$ and the Floer complex $CF_*^{[id]}((L,L),H;U)$
for which one needs to avoid bubbling (especially disc-bubbling) to ensure
the chain isomorphism property of the Floer-Piunikhin's continuation map. For the latter
purpose, we need to obtain some estimates of the filtration change for the
Floer chain map between the identity path and $\phi_H$ over the family
$$
\CH: s \mapsto H(s), \quad s \in [0,1].
$$

In the present context, we do not have such control over the filtration change
under the chain map we construct, \emph{even if one uses the adiabatic
chain map mentioned before}: Since we do not have any restriction on
the $C^2$-norm of $H$, we will not have much control on the mesh of the
partitions we make for the given approximating sequence $H_i$.
To overcome this lack of control of the filtration, we use Conley and Floer's idea of continuation
of maximal invariant sets \cite{conley:isolating,floer:fixed,oh:imrn},
which we now briefly summarize leaving more details thereto.

We denote by $\CM_1(J,(L',L);U)$ the set of pairs $(u,z)$ of $J$-holomorphic strips
$u$ attached to the pair $(L',L)$ whose image is contained in $U$ and a marked point
$z \in \R \times [0,1]$. We then denote
$$
\CS(J,(L',L);U) : = \overline{ev(\CM_1(J,(L',L);U))}
$$
and call it the maximal invariant set of the Cauchy-Riemann flow.
For a given one parametric family
$$
(J^{para},H^{para}) \in Map([0,1]^2,\CJ_\omega) \times C^\infty([0,1]^2 \times M,\R)
$$
with $H^{para} = \CH$ with $H(0) = 0$, we define a \emph{continuation}
$U^{para}$ between the maximal invariant sets $\CS_0 \subset U^0$ and $\CS_1 \subset U^1$
to be an open subset of $[0,1] \times M$ that satisfies
\begin{enumerate}
\item  For each $s \in [0,1]$ and all $t\in [0,1]$,
$$
L^s \subset U^s :=\{ x \in M \mid (x,s)\in U^{para} \}.
$$
\item
$$
\CS_s := \CS(J^s,(L^s,L);U^s)
$$
is isolated in $U^s$ for all $s \in [0,1]$.
\end{enumerate}

The following isolatedness
is a crucial ingredient in the construction of the isomorphism
$$
H_*(L) \cong HF^{[id]}(\phi_H^1(L),L), J;U).
$$
\begin{prop} $\CS(J_0,(\phi_{H(s)}^1(L),L);V)$ is isolated in $V$.
\end{prop}
\begin{proof} Consider the family $\CS_s := \CS(J_0,(\phi_{H(s)}^1(L),L);V)$
for $0 \leq s \leq 1$. Clearly $\CX_0 =  \CS(J_0,(L,L);V)$ is isolated in $V$.
Furthermore the isolatedness is an open property. Let $0< s_0 \leq 1$ be the smallest
$s$ at which $\CS_{s_0}$ fails to be isolated in $V$. Then there exists some
$z_0 = (\tau_0,t_0) \in \R \times [0,1]$
such that $v(\tau_0,t_0) \in \del \overline V$. Since $v(\tau,0) \in \phi_{H(s)}^1(L)$
and $v(\tau,1) \in L$ and $\phi_{H(s)}^1(L), \, L \subset V$, this violates the maximum
principle applied to the $J_0$-convex boundary of $V$. This finishes the proof.
\end{proof}

Once we have set up these definitions and isolatedness, it immediately gives rise to the following
theorem

\begin{thm}\label{thm:local2} Suppose $(L',L;J;U)$ is as above.
Suppose $H \in \CH_{\delta}^{engulf}(L;U)$ for a sufficiently small $\delta=\delta(\e) > 0$.
Then for any small perturbation $J'$ of $J$ for which $\MM(L',L;J';U)$ is Fredholm regular,
\begin{enumerate}
\item
the homomorphism
$$
\del _U: CF(L',L;J';U) \to CF(L',L;J';U),\quad \del_U x=\sum_{y \in L \cap \phi_H^1(L)}
\langle \del_U x, y\rangle y
$$
satisfies $\del_U \circ \del_U = 0$.
\item And the corresponding quotients
$$
HF(L,L;(H,J');U) \cong HF^*(L',L;J';U) = \ker\del_U/\operatorname{im} \del_U
$$
are isomorphic under the continuation $(\CS^{para}, J^{para}, H^{para}, U^{para})$
as long as the continuation is Floer-regular at the ends $s= 0, \, 1$.
\end{enumerate}
\end{thm}

After we establish this continuation invariance, we can apply it to the family $\CH$ with
$H(0) = 0$ and prove the following theorem. The proof of this theorem together with that of 
Corollary \ref{cor:rhoV10<E-F} may be the most novel part of the mathematics of the 
present paper beyond those already established in \cite{oh:imrn}, \cite{oh:jdg}.

\begin{thm}\label{thm:local-homology} Consider $\CH=\{H(s)\} \subset \CH^{engulf}_\delta(M)$
with $H(0) = 0$. Then whenever $0 < \delta < d(V,\Theta)$,
$$
HF^{[id]}(\phi_H^1(L),L;J';U) \cong H_*(L;\Z)
$$
for any $J'$ sufficiently close to $J_0$ in $C^\infty$-topology.
\end{thm}
\begin{proof}
We consider the homotopy
$$
\CH: s\mapsto H(s)
$$
and its reversal.
Using the isolatedness of thin trajectories in Theorem \ref{thm:thick-thin2} and
Theorem \ref{thm:handle2},  we define the local Floer-Piunikhin (pre)-chain maps
\beastar
\Psi_{\CH^\rho,(0)}: CF_*^{[id]}((L,L),0;U) & \to & CF_*^{[id]}((L,L),H;U), \\
\Phi_{\CH^{\widetilde \rho},(0)}: CF_*^{[id]}((L,L),H;U) & \to & CF_*^{[id]}((L,L),0;U)
\eeastar
and their compositions
\beastar
\Psi_{\CH^\rho,(0)}\circ \Phi_{\CH^{\widetilde \rho},(0)}:
CF_*((L,L),0;U) & \to & CF_*((L,L),H;U),\\
\Phi_{\CH^{\widetilde \rho},(0)} \circ \Psi_{\CH^\rho,(0)}: CF_*((L,L),H;U) & \to & CF_*((L,L),0;U).
\eeastar
Theorem \ref{thm:thick-thin2} and Theorem \ref{thm:handle2}
imply that all the above maps properly restrict to the maps
between $CF_*^{[id]}(0;U) \cong (C_*(L),\del_{(0)})$, the singular chain complex,
and $CF_*^{[id]}(H;U)$ by isolating the thin trajectories. Since the thin
trajectories cannot bubble-off, all these maps
become chain maps between them.
Therefore $\Psi_{\CH^\rho,(0)}$ and $\Phi_{\CH^{\widetilde \rho},(0)}$
induce the isomorphisms between $HF_*^{[id]}((L,L),0;U) \cong H_*(L)$ and $HF_*^{[id]}((L,L),H;U)$
which are inverses to each other.
More precisely, there exist a chain homotopy maps between
$\Psi_{\CH^\rho,(0)}\circ \Phi_{\CH^{\widetilde \rho},(0)}$ and
$id_{C_*(L)}$, and $\Phi_{\CH^{\widetilde \rho},(0)} \circ \Psi_{\CH^\rho,(0)}$ and $id_{CF_*(H)}$ respectively.
(See \cite{oh:newton,fooo:disc} for the proof of existence of such a chain homotopy.)
Once this is established, we can compute $HF_*^{[id]}((L,L),0;U)$ inside
the cotangent bundle $T^*L$. Then the arguments used in
\cite{floer:witten} and \cite{oh:imrn} prove the theorem.

This finishes the proof.
\end{proof}

\begin{rem} In the above proof, we would like to emphasize that no bubbling-off for
the thin trajectories holds not because the area will be big but because a bubble
must go out of the given Darboux neighborhood of $L$ and hence cannot be thin.
As we mentioned before we recall that thin trajectories could have large area.
\end{rem}

\section{Localization on the cotangent bundle}
\label{sec:cotangent}

The main purpose of this section is to use the local Floer complex
constructed on the cotangent bundle and localize the construction of
Lagrangian spectral invariants introduced in \cite{oh:jdg} which has been further studied in
\cite{oh:lag-spectral}.

We will also localize the triangle product similarly and the basic phase function
in the current context of approximations of engulfable topological
Hamiltonian loops in Appendix, for a future purpose.

\subsection{Localization of Lagrangian spectral invariants $\rho^{lag}(H;1)$}

We first specialize the general definition of spectral invariants
$\rho^{lag}(F;1)$ and $\rho^{lag}_V(F;1_0)$ to the cotangent bundle.
In this case of the Hamiltonian deformations of the zero section in
the cotangent bundle, we do not need to use the Novikov ring but only use
the coefficient ring $R$ and have only to use the single valued
classical action functional
$$
\CA_{F}^{cl}(\gamma) = \int \gamma^*\theta - \int_0^1
F(t,\gamma(t))\, dt
$$
in the evaluation of the level of the chains.

We now assume that $L$ is connected. Using the isomorphism
$$
(\Psi_{\CH})_*: H_*(L,\Lambda_{(\omega,L)}) \to HF_*(F)
$$
where $\Psi_{\CH}$ is the chain map defined in section \ref{sec:handle},
we define
$$
\rho^{lag}(F;1) = \inf_{\alpha \in (\Psi_\CH)_*([L])} \lambda_F(\alpha)
$$
which is also the same as
$$
\inf_\lambda \left\{HF_*^{[id],\lambda}(F) \neq 0\right\}.
$$
This is because $HF_n(F;V)$ or $HF_n(F)$ has rank one and so all
isomorphisms $H_*(L) \to HF_*(F)$ maps the fundamental cycle $[L]$ of $L$ to
the same image modulo a non-zero scalar multiple and so the associated
spectral invariants coincide (Confomality Axiom \cite{oh:alan}).
Similarly we define
$$
\rho^{lag}_V(F;1) = \inf_{\alpha \in (\Psi_{\CH,(0)})_*([L])} \lambda_F(\alpha)
$$
which is also the same as
$$
\inf_\lambda \left\{HF_*^{[id],\lambda}(F;V) \neq 0\right\}.
$$
\begin{rem}
We would like to mention that the homomorphism $(\Psi_\CH)_*$ and $(\Psi_{\CH,(0)})_*$
do not depend on the choice of homotopy $\CH$. But for the case $(\Psi_{\CH,(0)})_*$
the whole family of Hamiltonians $H(s)$ for $s \in [0,1]$ should be assumed to
be $V$-engulfable. For example, the commonly used the linear homotopy $s \mapsto s F$
may not be $V$-engulfalbe even when $F$ is $V$-engulfable. Because of this, the linear
homotopy cannot be used to construct the local chain map in general. Here is one place where
the presence of the engulfable family $\CH$ of Hamiltonians is used in the definition of
local Lagrangian spectral invariants.
\end{rem}

Now we prove the following coincidence theorem of global and local
spectral invariants.

\begin{thm} \label{thm:localrho=globalrho} Let $\CH=\{H(s)\}$ be a $V$-engulfable isotopy
with $H(0) = 0$ and $F = H(1)$. Then we have
$$
\rho^{lag}_V(F;1_0) = \rho^{lag}(F;1)
$$
\end{thm}
\begin{proof}
For the given family
$$
\CH: s \mapsto H(s), \quad s \in [0,1],
$$
we consider the continuation
of maximal invariant sets defined in section \ref{sec:cotangent}.
For given one parametric family
$$
(J^{para},H^{para}) \in Map([0,1]^2,\CJ_\omega) \times C^\infty([0,1]^2 \times M,\R)
$$
with $H^{para} = \CH$ with $H(0) = 0$ and $J^{para} = J_0$, all the Floer trajectories
contributing to these maximal invariant sets are thin and so become very
thin. In particular, the maximal invariant sets $\CS_s$ are all contained in the given
neighborhood $[0,1] \times D^\delta(T^*L)$ for all of $[0,1] \times M$
and $\CS_0 = L$.

This implies that the local Floer complex $(CF_*^{[id]}(F),\del_{(0)})$ and
the global one $(CF_*(F),\del)$ define the same complex and also satisfies
$$
(\Psi_{\CH})_*([L]) = (\Psi_{\CH,(0)})_*([L])
$$
under the identification, provided $F$ is connected to $0$ via an engulfable
Hamiltonian homotopy $\CH = \{H(s)\}$ is given. This finishes the proof.
\end{proof}

The proof of the following corollary requires some care unlike the
case of global Floer homology because the standard linear homotopy $s \mapsto sF$
may not be $V$-engulfable.

\begin{cor}\label{cor:rhoV10<E-F} For any $F \in \CH_\delta^{engulf}(T^*L;V)$,
$\rho^{lag}_V(F;1_0) \leq E^-(F)$.
\end{cor}
\begin{proof} Knowing that Theorem \ref{thm:localrho=globalrho} holds, we can consider
the linear homotopy $s \mapsto sF$ and denote by
$\Psi_F^{lin}$ the associated Floer-Piunikhin chain map $C_*(L) \to
CF_*(F)$ for the \emph{global} Floer complex instead. Then it is standard
that $\Psi_F^{lin}$ also induces an isomorphism
$H_*(L) \to HF_*(F)$ in global Floer homology. More specifically we have
$$
(\Psi_F^{lin})_*([L]) = (\Psi_\CH)_*([L]).
$$
(We emphasize that the corresponding cycles $(\Psi_F^{lin})_\#([L])$,
$(\Psi_\CH)_\#([L])$ are different in general. For example, the general estimate of
the level of the cycle $(\Psi_\CH)_\#([L])$ involve the derivative $\frac{\del H(s)}{\del s}$
which is uncontrolled in the topological Hamiltonian homotopy.)
Using the cycle $(\Psi_F^{lin})_\#([L])$, it is easy to obtain the upper bound
$\rho(F;1) \leq E^-(F)$ by the standard calculations. (See  \cite{oh:jdg,oh:ajm1,oh:alan} for example).
This together with Theorem \ref{thm:localrho=globalrho}
gives rise to the proof.
\end{proof}

\begin{rem}\label{rem:linearhomotopy}
 We would like to emphasize that unlike the isotopy $\CH = \{H(s)\}$
with $H(s) \in \CH_\delta^{engulf}(L;V)$, the isotopy of the time-one maps $\phi_{sF}^1$
for the linear isotopy $\CH^{lin}: s \mapsto sF$ with $F = H(1)$ may not be
uniformly $C^0$-small and hence the associated Floer trajctories of the chain map moduli space could go
out of the neighborhood $V$. Because of this, the linear isotopy cannot be used
to define a chain map from $C_*(L)$ to the local Floer complex $CF^{[id]}_*(F;V)$
and so the inequality stated in this corollary does not follow from by now the
standard computation used in \cite{oh:alan} to prove $\rho^{lag}(F;1_0) \leq E^-(F)$ for the global
invariant.
\end{rem}

\subsection{Localization of the basic phase function}

We consider the Lagrangian pair
$$
(o_N, T^*_qN), \quad q \in N
$$
and its associated Floer complex $CF(H;o_N, T^*_qN)$ generated by
the Hamiltonian trajectory $z:[0,1] \to T^*N$ satisfying
\be\label{eq:Hamchordeq} \dot z = X_H(t,z(t)), \quad z(0) \in o_N,
\, z(1) \in T^*_qN. \ee Denote by $\CC hord(H;o_N,T^*_qN)$ the set of
solutions. The differential $\del_{(H,J)}$ on $CF(H;o_N, T^*_qN)$ is
provided by the moduli space of solutions of the perturbed
Cauchy-Riemann equation
\be\label{eq:CRHJq}
\begin{cases}
\dudtau + J\left(\dudt - X_H(u) \right) = 0 \\
u(\tau,0) \in o_N, \, u(\tau,1) \in T^*_qN.
\end{cases}
\ee

An element $\alpha \in CF(H;o_N,T^*_qN)$ is expressed as a finite
sum
$$
\alpha = \sum_{z \in \CC hord(H;o_N,T_q^*N)} a_z [z], \quad a_z \in
\Z.
$$
We denote the level of the chain $\alpha$ by
\be\label{eq:lambdaH}
\lambda_H(\alpha): = \max_{z \in \supp \alpha} \{\CA^{cl}_H(z)\}.
\ee
The resulting invariant $\rho(H;\{q\})$ is to be defined by the mini-max
value
$$
f_H(q): = \inf_{\alpha \in [q]}\lambda_H(\alpha)
$$
where $[q] \in H_0(\{q\};\Z)$ is a generator of the homology group
$H_0(\{q\};\Z)$.

Equivalently, we can consider the pair $(L_H,T_q^*N)$ for the action functional
$$
\CA^{(0)}(\gamma):= \int \gamma^*\theta + h_H(\gamma(0))
$$
defined on $\Omega(L_H,T_q^*N)$ which defines the geometric version of
the Floer complex $CF(L_H,T_q^*N)$ via the equation
\be\label{eq:dvdJ0}
\begin{cases}
\dvdtau + J_0 \dvdt = 0 \\
v(\tau,0) \in L_H, \, v(\tau,1) \in T^*_qN.
\end{cases}
\ee

Now by the same argument performed in sections \ref{sec:thick-thin} and \ref{sec:handle},
we can localize the Floer complex to $CF(L_H,T_q^*N;V)$ and define the
local version of the spectral invariant $\rho^{lag}_V(H;\{q\})$ by
$$
f_H^V(q) = \inf_{\alpha \in [q]}\lambda_H(\alpha)
$$
where $[q] \in H_0(\{q\};\Z)$ is a generator of the homology group
$H_0(\{q\};\Z)$.  By varying $q \in N$, this
defines a function $f_H^V: N \to \R$ which is precisely the local version of
the basic phase function defined in \cite{oh:jdg}.
We denote the associated graph part of the front $W_{R_H}$ of the $L_H$ by $G_{f_H^V}$.

We summarize the main properties of $f_H^V$ whose proofs are verbatim the same
as those established for the (global) basic function $f_H$ in \cite{oh:jdg},
\cite{oh:lag-spectral} by replacing the global Floer complex $CF_*(H)$ by the
local complex $CF_*^{[id]}(H;V)$. First we have

\begin{thm}
Let $H=H(t,x) \in \CH_\delta^{engulf}(T^*N)$ and the Lagrangian
submanifold $L_H = \phi_H^1(o_N)$. Consider the function $f_H^V$ defined above.
Then for any $x \in L_H$
\be\label{eq:fversush}
f_H^V(\pi(x)) = h_H(x) = \CA^{cl}_H(z_x^H)
\ee
for some Hamiltonian chord $z_x^H$ ending at $L_H \cap T_{\pi(x)}^*N$.
\end{thm}

Once we have achieved localizations of various entities arising in
Floer complex in the previous subsection, the following equality can be proven
by the same argument used in the proof of Theorem \ref{thm:localrho=globalrho}
using the localized version of Lagrangian spectral invariants and
basic phase function. We omit the details of its proof.

\begin{thm}\label{thm:localfH=globalfH} Let $V \subset T^*N$ be as before. Then
$$
f_H^V = f_H
$$
for any $V$-engulfable $H$.
\end{thm}

An immediate corollary of this theorem is the following inequality.

\begin{cor}\label{cor:rhoVH>fHV}
For any Hamiltonian $H \in \CH_\delta^{engulf}(T^*N)$,
\be\label{eq:fVH<rhoVH}
\max f_H^V \leq E^-(H).
\ee
Furthermore if $H, \, H' \in \CH_\delta^{engulf}(T^*N)$
\be\label{eq:fH}
\|f_H^V - f_{H'}^V\|_\infty \leq \|H - H'\|.
\ee
\end{cor}

\subsection{Localization of triangle product}
\label{subsec:triangle}

A version of localization of triangle product was previously
exploited in \cite{seidel:triangle,spaeth,oh:seidel} for smooth Hamiltonians.

Instead of delving into the localization of triangle product in full generality,
we will restrict ourselves to the case of the zero section $o_L$
in the cotangent bundle. Once we isolate the invariant set into a Darboux neighborhood $U
\subset M$, we may identify $U$ with a neighborhood $V$ of the zero
section $o_L \subset T^*L$ and consider a Hamiltonian $F$ with
$\supp F \subset V$. It then follows that due to the non-presence
of bubbling effect for the pair $(T^*L, o_L)$, by an easier
argument, we obtain the decomposition $\del = \del_{(0)} + \del'$ of the
Floer differential $\del$ on $CF_*(F;T^*L)$, and obtain the local
Floer complex
$$
\left(CF_*^{[id]}(o_L,F;V), \del_{(0)}\right).
$$
We first recall the definition of the triangle product described in
\cite{oh:cag}, \cite{foh:ajm} and the discussion carried out in
section 8 \cite{oh:lag-spectral}. Similar idea of localizing the triangle
product was used in \cite{seidel:triangle}, \cite{oh:seidel} and \cite{spaeth}.
Instead of delving into the localization in full generality, we restrict ourselves
to the case relevant to our main interest arising from the study in \cite{oh:lag-spectral}.

Let $q \in N$ be given. Consider the Hamiltonians $H: [0,1] \times
T^*N \to \R$ such that $L_H$ intersects transversely both $o_N$ and
$T_q^*N$. We consider the Floer complexes
$$
CF(L_H,o_N), \quad CF(o_N,T_q^*N), \quad CF(L_H,T_q^*N)
$$
each of which carries filtration induced from the effective action
function given below. We denote by
$\frak v(\alpha)$ the level of the chain $\alpha$ in any of these
complexes.

More precisely, $CF(L_H,o_N)$ is filtered by the effective
functional
$$
\CA^{(1)}(\gamma):= \int \gamma^*\theta + h_H(\gamma(0)),
$$
$CF^{\mu}(o_N,T_q^*N)$ by
$$
\CA^{(2)}(\gamma) := \int \gamma^*\theta,
$$
and $CF(L_H,T_q^*N)$ by
$$
\CA^{(0)}(\gamma):= \int \gamma^*\theta + h_H(\gamma(0))
$$
respectively. We recall the readers that $h_H$ is the potential of
$L_H$ and the zero function the potentials of $o_N, \, T_q^*N$.

We now consider the triangle product in the chain level, which we
denote by
\be\label{eq:m2chain} \frak m_2: CF(L_H,o_N) \otimes
CF(o_N,T_q^*N) \to CF(L_H,T_q^*N)
\ee
ollowing the general notation
from \cite{fooo:book}. This product is
defined by considering all triples
$$
x_1 \in L_H \cap o_N, \, x_2 \in o_N \cap T_q^*N, \, x_0 \in L_H
\cap T_q^*N
$$
with the polygonal Maslov index $\mu(x_1,x_2;x_0)$ whose associated
analytical index, or the virtual dimension of the moduli space
$$
\CM_{3}(D^2; x_1,x_2; x_0): = \widetilde
\CM_3(D^2;x_1,x_2;x_0)/PSL(2,\R)
$$
of $J$-holomorphic triangles, becomes zero and counting the number
of elements thereof.

\begin{defn}\label{defn:CM3} Let $J = J(z)$ be a domain-dependent family of
compatible almost complex structures with $z \in D^2$. We define the
space $\widetilde \CM_3(D^2;x_1,x_2;x_0)$ by the pairs
$(w,(z_0,z_1,z_2))$ that satisfy the following:
\begin{enumerate}
\item $w: D^2 \to T^*N$ is a continuous map satisfying $\delbar_J w = 0$ on
$D^2 \setminus \{z_0,z_1,z_2\}$,
\item the marked points $\{z_0,z_1,z_2\} \subset \del D^2$ with counter-clockwise
cyclic order,
\item $w(z_1) = x_1, \, w(z_2)=x_2$ and $w(z_0) = x_0$,
\item the map $w$ satisfies the Lagrangian boundary condition
$$
w(\del_1 D^2) \subset L_H, \, w(\del_2 D^2) \subset o_N, \, w(\del_3
D^2) \subset T_q^*N
$$
where $\del_i D^2 \subset \del D^2$ is the are segment in between
$x_i$ and $x_{i+1}$ ($i \mod 3$).
\end{enumerate}
\end{defn}

We have the following energy estimate

\begin{prop}[Proposition 8.2 \cite{oh:lag-spectral}] Suppose $w:D^2 \to T^*N$ be any smooth map with finite energy
that satisfy all the conditions given in \ref{defn:CM3}, but not
necessarily $J$-holomorphic. We denote by $c_x:[0,1] \to T^*N$ the
constant path with its value $x \in T^*N$. Then we have
\be\label{eq:area} \int w^*\omega_0 = \CA^{(1)}(c_{x_1}) +
\CA^{(2)}(c_{x_2}) - \CA^{(0)}(c_{x_0}) \ee
\end{prop}

An immediate corollary of this proposition from the definition of
$\frak m_2$ is that the map \eqref{eq:m2chain} restricts to
$$
\frak m_2: CF^{\lambda}(L_H,o_N) \otimes CF^{\mu}(o_N,T_q^*N) \to
CF^{\lambda + \mu}(L_H,T_q^*N)
$$
and in turn induces the product map
\be\label{eq:m2hom}
*_F: HF^{\lambda}(L_H,o_N) \otimes HF^{\mu}(o_N,T_q^*N) \to HF^{\lambda+\mu}(L_H,T_q^*N)
\ee
in homology. This is because if $w$ is $J$-holomorphic $\int
w^*\omega \geq 0$. This ends the summary of
triangle product on the global Floer complex explained in
\cite{oh:lag-spectral}.

To localize the above construction to obtain the local analogs
$$
\frak m_{2,(0)}: CF^{\lambda}(L_H,o_N;V) \otimes CF^{\mu}(o_N,T_q^*N;V) \to
CF^{\lambda + \mu}(L_H,T_q^*N;V)
$$
and the induced product
\be\label{eq:m2hom}
*_{F,(0)}: HF^{\lambda}(L_H,o_N;V) \otimes HF^{\mu}(o_N,T_q^*N;V) \to HF^{\lambda+\mu}(L_H,T_q^*N;V)
\ee
in homology, we have only to prove the analog to Theorem \ref{thm:thick-thin2}
and Theorem \ref{thm:handle2} for the moduli space
$$
\widetilde \CM_3(D^2;x_1,x_2;x_0).
$$
\begin{thm} Let $V$ be an open neighborhood of the zero section $o_L$ and
let $H \in \CH^{engulf}_\delta(T^*L)$. Then for any given open neighborhood $V$ of
$o_L$, there exists some $\delta_0 > 0$ such that
for any $0 < \delta \leq \delta_0$, for any element
$w \in \widetilde \CM_3(L_H, o_N, T_q^*N)$
the following alternative holds:
\begin{enumerate}
\item $\Image w \subset V$ and $\max_{z \in \R \times \{0,1\}}d(v(z), o_L) \leq \delta$,
\item $\Image w \not \subset V$ and $\int w^*\omega \geq C(J_0,V)$.
\end{enumerate}
\end{thm}
\begin{proof} The only difference in the proof of this theorem from
Theorem \ref{thm:thick-thin2} and \ref{thm:handle2} is that we also need to use
the strong maximum principle along the fiber Lagrangian $T_q^*N$ in addition.
We would like to note that the intersection
$$
T_q^*N \cap S^\delta(T^*N)
$$
is Legendrian and so a $J_g$-holomorphic curve satisfies strong
maximum principle along $T_q^*N$. We refer to \cite{EHS}, \cite{oh:jdg}
for such an application of strong maximum principle to obtain $C^0$-estimate.
\end{proof}

The proof is exactly the same as that of Theorem \ref{thm:thick-thin2}
and Theorem \ref{thm:handle2} and so omitted.

We define the `thin' part of
$\frak m_2$ by counting those elements $w$ from $\widetilde \CM_3(L_H, o_N, T_q^*N)$
of the type (1) above and decompose
$$
\frak m_2 = \frak m_{2,(0)} + \frak m_2'.
$$
It also follows that $\frak m_{2,(0)}$ induces a product map
$$
\frak m_{2,(0)}:CF^{\lambda}(L_H,o_N;V) \otimes CF^{\mu}(o_N,T_q^*N;V) \to
CF^{\lambda + \mu}(L_H,T_q^*N;V).
$$
It is straightforward to check that this map satisfies
$$
\del_{(0)}(\frak m_{2,(0)}(x,y)) = \frak m_{2,(0)}(\del_{(0)}(x),y) \pm \frak m_{2,(0)}(x,\del_{(0)}(y))
$$
and so induces a product
\be\label{eq:m2hom}
*_{F,(0)}: HF_*^\lambda(L_H,o_N;V) \otimes HF_*^\mu(o_N,T_q^*N;V)
\to HF_*^{\lambda+\mu}(L_H,T_q^*N;V))
\ee
as in \cite{oh:lag-spectral}.

\section{Appendix: Local Floer complex of engulfable
Hamiltonian $C^0$-approximate loop}

In this appendix, we give the construction of local Hamiltonian Floer complex in the
context of $C^0$-small topological Hamitonian loops for a future purpose.
Exposition of this appendix closely follows that of section 4 \cite{oh:ajm1}
except that we need to explain the points, if necessary, about why $C^0$-smallness of
$\phi_F$ is enough to localize the Floer complex of the fixed point set of
$\phi_F^1$.

\subsection{Hamiltonian Floer complex}
\label{subsec:pert-action}

This section reviews the standard construction in Hamiltonian Floer theory.
We closely follow exposition of chapter 2 \cite{fooo:spectral} for some
enhancement added which is useful for our purpose later.

Let $\widetilde{\CL}_0(M)$ be the set of all the pairs $[\gamma,w]$
where $\gamma$ is a loop $\gamma: S^1 \to M$ and $w: D^2 \to M$ a
disc with $ w\vert_{\del D^2} = \gamma$. We identify $[\gamma,w]$
and $[\gamma',w']$ if $\gamma = \gamma'$ and  $w$ is homotopic to
$w'$ relative to the boundary $\gamma$. When a one-periodic
Hamiltonian $H:(\R/\Z) \times M \to \R$ is given, we consider the
perturbed functional $\CA_H: \widetilde \CL_0(M) \to \R$ defined by
\be\label{eq:AAH}
\CA_H([\gamma,w]) = -\int w^*\omega - \int
H(t,\gamma(t))dt.
\ee

For a Hamiltonian $H:[0,1] \times M \to \R$,
we denote its flow, a Hamiltonian isotopy, by $\phi_H: t\mapsto
\phi_H^t \in \text{\rm Ham}(M,\omega)$. We
denote the time-one map by $\phi_H^1$. We put
$$
\Fix \phi_H^1 = \{ p \in M \mid \phi^1_H(p) = p\}.
$$

Each element $p\in \mbox{\rm Per}(H)$, the set of $1$-periodic orbits, induces a map
$
z_x=z_x^H : S^1 \to M,
$
by the correspondence
\be\label{eq:Fix-Per}
z^H_x(t) = \phi_H^t(\phi_H^{-1}(x)),
\ee
where $t \in \R/\Z \cong S^1$.

We denote by $\mbox{\rm Per}(H)$ the set of one-periodic solutions of $\dot x = X_H(t,x)$.
Then \eqref{eq:Fix-Per} provides a one-one correspondence between $\operatorname{Fix}\phi_H^1$ and
$\operatorname{Per}(H)$.
The set of critical points of $\CA_H$ is given by
$$
\mbox{\rm Crit}(\CA_H) = \{ [z,w] \mid \gamma \in \mbox{\rm Per}(H), \, w\vert_{\del D^2}  = \gamma \}.
$$

We consider the universal (downward) Novikov field
$$
\Lambda = \left\{\sum_{i=1}^{\infty} a_i T^{\lambda_i} \Big|\, a_i \in \R, \, \lambda_i \to -\infty\right\}
$$
and define a valuation $\frak v_T$ on $\Lambda$ by
\begin{equation}\label{vqdef}
 \frak v_T\left(\sum_{i=1}^{\infty} a_i T^{\lambda_i}\right)
 = \sup \{ \lambda_i \mid a_i \ne 0 \}.
\end{equation}
It satisfies the following properties:
\begin{enumerate}
\item $ \frak v_T(xy) =  \frak v_T(x) +  \frak v_T(y)$,
\item $ \frak v_T(x+y) \le \max\{ \frak v_T(x),  \frak v_T(y)\}$,
\item  $\frak v_T(x) = - \infty$ if and only if $x =0$,
\item $\frak v_T(q) = 1$,
\item $\frak v_T(ax) = \frak v_T(x)$ if $a \in R \setminus \{0\}$.
\end{enumerate}

We consider the
$\Lambda$ vector space
$\widehat{CF}(H;\Lambda)$ with basis given by the critical point set $\mbox{Crit}(\CA_H)$ of $\CA_H$.
\begin{defn}\label{Lambdahatonashi}
We define an equivalence relation $\sim$ on $\widehat{CF}(H;\Lambda)$ so that
$[z,w] \sim T^c[z',w']$ if and only if
\be\label{eq:equiv-rel}
z= z', \,\, \int_{D^2} w^{\prime *}\omega = \int _{D^2} w^*\omega - c.
\ee
\par
The quotient of $\widehat{CF}(H;\Lambda)$ modded out by this equivalence relation $\sim$
is called the Floer complex of the periodic Hamiltonian $H$ and denoted by
${CF}(H;\Lambda)$.
\end{defn}

Here we do not assume the condition on the Conley-Zehnder indices and work with $\Z_2$-grading.
In the standard literature on Hamiltonian Floer homology, an additional requirement
$$
c_1(\overline w \# w') = 0
$$
is commonly imposed in the definition Floer complex, denoted by $CF(H)$. For the purpose of
the current paper similarly as in \cite{fooo:spectral}, the equivalence relation \eqref{eq:equiv-rel}
is enough and more favorable in that it makes the associated Novikov ring becomes a field.
To differentiate the current definition from $CF_*(H)$, we denote the complex
used in the present paper by ${CF}_*(H;\Lambda)$.

\begin{lem}
As a $\Lambda$ vector space, ${CF}_*(H;\Lambda)$ is isomorphic to
the direct sum $\Lambda^{\# \mbox{\rm Per}(H)}$.
\par
Moreover the following holds: We fix a lifting $[z,w_z]
\in \mbox{\rm Crit}(\CA_H)$ for each
$z \in \mbox{\rm Per}(H)$. Then any element $x$ of ${CF}(M,H;\Lambda)$
is uniquely written as a sum
\begin{equation}\label{xexpand}
x = \sum_{z \in \mbox{\rm Per}(H)} x_z [z,w_z], \quad
\text{with $x_z \in \Lambda$}.
\end{equation}
\end{lem}

\begin{defn}\label{defn:valuationv}
\begin{enumerate}
\item
Let $x$ be as in $(\ref{xexpand})$. We define
$$
\frak v_T(x) = \max \{ \frak v_T(x_z) + \CA_H([z,w_z]) \mid \gamma \in \mbox{\rm Per}(H) \}.
$$
\item
We define a filtration $F^{\lambda}CF(M,H;\Lambda)$ on $CF(M,H;\Lambda)$ by
$$
F^{\lambda}CF(H;\Lambda)
= \left\{ x \in CF(H;\Lambda) \mid \frak v_T(x) \le \lambda\right\}.
$$
We have
$$
F^{\lambda_1}CF(H;\Lambda)  \subset F^{\lambda_2}CF(H;\Lambda)
$$
if $\lambda_1 < \lambda_2$. We also have
$$
\bigcap_{\lambda} F^{\lambda}CF(H;\Lambda)
=\{0\},
\quad
\bigcup_{\lambda} F^{\lambda}CF(H;\Lambda)
= CF(M;H).
$$
\item
We define a metric $d_T$ on $CF(H;\Lambda)$ by
\be\label{eq:metricdq}
d_T(x,x') = e^{\frak v_T(x-x') }.
\ee
\end{enumerate}
\end{defn}
Then (\ref{vqdef}), (\ref{eq:equiv-rel}) and Definition \ref{defn:valuationv}
imply that
$$
\frak v_T(a\frak x) = \frak v_T(a) + \frak v_T(\frak x)
$$
for $a \in \Lambda^{\downarrow}$,
$\frak x \in CF(H;\Lambda)$.
We also have
$$
T^{\lambda_1}\cdot F^{\lambda_2}CF(H;\Lambda)
\subseteq
F^{\lambda_1+\lambda_2}CF(H;\Lambda).
$$

\begin{lem}\label{lem37}
\begin{enumerate}
\item
$\frak v_T$ is independent of the choice of the lifting $z \mapsto [z,w_z]$.
\item
$CF(H;\Lambda^{\downarrow})$ is complete with respect to the metric $d_T$.
\item
The infinite sum
$$
\sum_{[z,w] \in \Crit \CA_H} x_{[z,w]} [z,w]
$$
converges in $CF(H;\Lambda^{\downarrow})$ with respect to the metric $d_T$ if
$$
\left\{ [z,w] \in \Crit\CA_H \mid \frak v_T(x_{[z,w]}) +   \mathcal A_H([z,w]) > -C,
\,\, x_{[z,w]} \ne 0 \right\}.
$$
is finite for any $C \in \R$.
\end{enumerate}
\end{lem}

\subsection{Isolating local Hamiltonian Floer complex}

This section is a modification of section 4.1 \cite{oh:ajm1} which
treats the case of $C^2$-small perturbation of Hamiltonians $H$
following  section 3 \cite{oh:imrn}.

As in section \ref{sec:local},
we will replace the condition of $\phi_F$ being $C^1$-small by $\phi_F$ being $C^0$-small with the
same kind of bound on the Hofer norm $\|F\|$.
Once we have established the thick-thin decomposition
given in Theorem \ref{thm:thick-thin2}, we can safely repeat the arguments
laid out in section 4.1 \cite{oh:ajm1}, whose summary is now in order.

For given such $F$, we consider the subset $\CU = \CU(U_\Delta) \subset \CL_0(M)$
of loops given by
$$
\CU = \{\gamma \in \CL_0(M) \mid (\gamma(t),\gamma(0)) \in U_\Delta\}.
$$
for a fixed Darboux neighborhood $U_\Delta$ of the diagonal $\Delta \subset M\times M$
for all $t \in [0,1]$. In particular, any periodic orbit $z$ of the flow
$\phi_H$ is contained in $\CU \subset \CL(M)$ has a canonical isotopy
class of contraction $w_z$. We will always use this convention
$w_z$ whenever there is a canonical contraction of $z$ like in
this case of small loops. This provides a canonical embedding of
$\CU \subset \widetilde \CL_0(M)$ defined by
$$
z \to [z,w_z].
$$
We denote this canonical embedding by $\CU^{[id]}$. This selects a
distinguished component of
$$
\pi^{-1}(\CU) \subset  \widetilde \CL_0(M)
$$
and other components can be given by
$$
\CU^{[g]} = g \cdot \CU^{[id]}, \quad g \in \Gamma_\omega
$$
similarly as before.

Combining the constructions from \cite{oh:ajm1} and section \ref{sec:local}, we give

\begin{defn} Let $J=\{J_t\}$ with $|J_t-J_0|_{C^1} < \e_3$ with $\e_3$ sufficiently small.
For any $F \in \CH_{\delta}^{engulf}(M)$ and for the given Darboux neighborhood $U_\Delta$ of
the diagonal $\Delta \subset M\times M$ such that
$$
\phi_F^t(\Delta) \subset \hbox{\rm Int }U_\Delta,
$$
we define
$$
\MM^{[g]}(F,J; \CU) = \{u \in \MM(F,J) \mid (u(\tau)(t), u(\tau)(0))
\in \hbox{\rm Int } U_\Delta^{[g]}  \, \hbox{\rm for all }\, \tau \}
$$
for each $g \in \Gamma_\omega$. Consider the evaluation map
$$
ev:\MM\left(F,J:\CU^{[g]}\right) \to \CU \subset \CL_0(M); \quad ev(u) = u(0).
$$
For each open neighborhood $U_\Delta \subset M\times M$ of
$\Delta \subset U_\Delta$,
we define the {\it local Floer complex} in $\CU^{[g]}$ by
$$
\CS\left(F,J; \CU^{[g]}\right) : = ev\left(\MM(F,J; \CU^{[g]}\right) \subset \CL_0 (M).
$$
We say $\CS(F,J; \CU^{[g]})$ is {\it isolated} in $\CU^{[g]}$
if its closure is contained in $\CU^{[g]}$.
\end{defn}

Using Theorem \ref{thm:local2}, we define the local Floer
homology, denoted by $HF^{[g]}(F,J;\CU)$. Furthermore, the pull-back
of the action functional $\CA_F$ to $\CU^{[g]}$ via the above mentioned
embedding into $\widetilde\CL_0(M)$ provides a filtration on the local Floer complex
$CF^{[g]}(F;\CU)$.

Therefore by considering the parameterized family
$$
\CS(G^s,J; \CU^{[id]}),
$$
the proof of Theorem \ref{thm:local2} implies that if $G \in \CH_{\delta}^{engulf}(M)$ and
$\delta$ sufficiently small, $\CS(J, G^s:\CU^{[id]})$
are isolated in $\CU^{[id]}$ for all $s$ and its homology is isomorphic to
$H_*(M;R)$. For readers' convenience, we provide the detailed comparison argument between
the Hamiltonian Floer complex of $\Fix \phi_G^1$ and the Lagrangian
Floer complex of the pair $(\Delta, \Graph \phi_G^1)$ in Appendix
borrowing from that of section 4.2 \cite{oh:ajm1}.

\subsection{ $\text{Fix }\phi^1_G$ versus $\Delta \cap \text{graph
}\phi^1_G$}

The main goal of this sub-section is to compare the Hamiltonian Floer homology of $G$
with the Lagrangian Floer complex between
$\Delta$ and $\text{graph }\phi_G^1$ in the product $(M, \omega)
\times (M,-\omega)$ when $G \in \CH_{\delta}^{engulf}(M)$
with $\delta$ sufficiently small.

We now compare the local Floer homology $HF^{[id]}(J,G: \CU)$ of
$G \in \CH_{\delta}^{engulf}(M)$ and two versions of its intersection
counterparts, one  $HF_{J_0\oplus -J_0, 0}^{[id]}(\Graph\phi^1_G, \Delta:U_\Delta)$
and the other $HF_{(\phi_G)^*J_0\oplus -J_0, 0\oplus
G}^{[id]}(\Delta, \Delta:U_\Delta)$.

First we note that the two Floer complexes $\MM_{J_0\oplus -J_0,
0}(\Graph\phi^1_G, \Delta:\CU^{[id]})$ and $\MM_{
(\phi_G)^*J_0\oplus -J_0, 0\oplus G}(\Delta, \Delta:\CU^{[id]})$ are canonically
isomorphic by the assignment
$$
(\gamma(t),\gamma(t)) \mapsto
\left((\phi^t_G)^{-1}(\gamma)(t),\gamma(t)\right).
$$
and so the two Lagrangian intersection Floer homology are
canonically isomorphic: Here the above two moduli spaces are the
solutions sets of the following Cauchy-Riemann equations
$$
\begin{cases} \frac{\del U}{\del \tau} +
(J_0\oplus -J_0) \frac{\del U}{\del t}  = 0 \\
 U(\tau,0) \in \text{graph }\phi^1_G, \, U(\tau, 1) \in \Delta
\end{cases}
$$
and
$$
\begin{cases} \frac{\del U}{\del \tau} +
((\phi_G^1)^*J_0) \oplus (-J_0)  \Big(\frac{\del U}{\del t}
- X_{0\oplus G}(U)\Big) = 0 \\
U(\tau, 0) \in \Delta, \, U(\tau,1) \in \Delta
\end{cases}
$$
respectively, where $U = (u_1,u_2) : \R \times [0,1] \to M\times
M$. The relevant action functionals for these cases are given by
\be\label{eq:AA0}
\CA_0([\Gamma, W]) = - \int W^*(\omega \oplus -\omega)
\ee
on $\widetilde \Omega(\Graph\phi^1_G, \Delta: M \times M)$
and
\be\label{eq:AAG}
\CA_{0\oplus G}([\Gamma, W]) = \CA_0 (\Gamma, W)- \int_0^1
(0\oplus G)(\Gamma(t), t)\, dt
\ee
on $\widetilde \Omega(\Delta, \Delta: M \times M)$ where we denote
$$
\Omega(\Graph\phi^1_G, \Delta: M \times M) = \{ \Gamma:
[0,1] \to M\times M \mid \Gamma(0) \in
\text{graph }\phi_G^1, \, \Gamma(1) \in \Delta, \}
$$
and similarly for $\Omega(\Graph\phi^1_G, \Delta: M \times
M)$. Again the `tilde' means the  covering space which can be
represented by the set of pairs $[\Gamma, W]$ in a similar way.
The relations
between the action functionals \eqref{eq:AA0}, \eqref{eq:AAG} and
$\CA_G$ are evident and respect the filtration under the natural
correspondences.

Next we will attempt to compare
$$
HF^{[id]}(G,J;\CU), \quad HF_{J_0\oplus -J_0,
G\oplus 0}^{[id]}(\Delta, \Delta:U_\Delta).
$$
Without loss of any generality,
we will concern Hamiltonians $G$ such that $G\equiv 0$ near $t =
0,\, 1$, which one can always achieve by perturbing $G$ without
changing its time-one map.

\emph{There is no direct way of identifying the
corresponding Floer complexes between the two}.

As an intermediate case, we consider the Hamiltonian $G^\prime: M
\times [0,1]$ defined by
$$
G^\prime(x,t) = \begin{cases} 2 G(x,2t)   \quad  & \text{for }\, 0\leq t \leq \frac1{2}\\
0 & \text{for } \, \frac1{2} \leq t \leq 1
\end{cases},
$$
and the assignment
\be\label{eq:diffG'}
(u_0,u_1) \in \MM_{J_0\oplus -J_0, G\oplus 0}^{[id]}(\Delta, \Delta:U_\Delta)
\mapsto v \in \MM(J, G^\prime: \CU^{[id]})
\ee
with $v(\tau,t) := u_0\# \overline u_1(\tau,t)$. Here the map
$u_0 \# \overline u_1: [0,1] \to M$ is the map defined by
$$
u_0 \# \overline u_1(\tau,t) =
\begin{cases} u_0(2\tau,2t) \quad & \text{for }\, 0\leq t \leq \frac{1}{2} \\
u_1(2\tau,1-2t) & \text{for } \, \frac{1}{2}\leq t \leq 1
\end{cases}
$$
is well-defined and continuous because
\beastar
u_0(\tau, 1) & = & u_0(\tau,0) = \overline u_1(\tau,0)\\
\overline u_1(\tau,1) & = & u_0 (\tau,1) = u_0(\tau,0).
\eeastar
Furthermore near $t = 0, \, 1$, this is smooth (and so
holomorphic) by the elliptic regularity since $G^\prime$ is smooth
(Recall that we assume that $G \equiv 0$ near $t =0,\, 1$.
Conversely, any element $v \in \MM(J,G^\prime:\CU^{[id]})$ can be written
as the form of $u_0 \# \overline u_1$ which is uniquely determined
by $v$. This proves that \eqref{eq:diffG'} is a diffeomorphism from
$\MM_{J_0\oplus -J_0, G\oplus 0}^{[id]}(\Delta, \Delta:U_\Delta)$ to $\MM(J,
G^\prime: \CU^{[id]})$ which induces a filtration-preserving isomorphism
between $HF_{J_0\oplus -J_0, G\oplus 0}^{[id]}(\Delta, \Delta:U_\Delta)$ and
$HF(J, G^\prime: \CU^{[id]})$.

Finally, we need to relate $HF(J,G:\CU^{[id]})$ and $HF(J, G^\prime:\CU^{[id]})$.
For this we note that $G$ and $G^\prime$ can be connected by a one-parameter
family $G^{para} = \{G^s\}_{0\leq s\leq 1}$ with
$$
G^s(x,t):=
\begin{cases} \frac2{1+s}G(x,\frac2{1+s}t)  \quad &\text{for }\, 0\leq t \leq \frac{s}{2} \\
0 & \text{for } \frac{s}{2} \leq t \leq
1.
\end{cases}
$$
And we have
$$
\phi^1_{G^s} = \phi^1_G\quad \text{for all $s \in [0,1]$}.
$$
Therefore their spectra coincide, i.e., $\Spec(G) = \Spec(G^{\prime s}) = \Spec(G')$.
Then there exists an isomorphism
$$
h^{adb}_{G^{para},J}: CF(G^\prime :\CU^{[id]}) \to  CF(G :\CU^{[id]})
$$
respects the filtration and so the induced homomorphism in its
homology
$$
h^{adb}_{G^{para},J}: HF(J,G^\prime:\CU^{[id]}) \to  HF(J,G :\CU^{[id]})
$$
becomes a filtration-preserving isomorphism. See \cite{kerman}, \cite{usher:depth},
\cite{oh:book} for such a construction.

\end{document}